\documentclass[11pt]{amsart}
\usepackage[utf8]{inputenc}
\usepackage{amsfonts, amssymb, amsmath, amsthm, color, float,enumerate}
\usepackage{url}
\usepackage{bm}
\usepackage[unicode,psdextra]{hyperref}
\usepackage{pgfplots,tikz}
\pgfplotsset{compat=1.18}

\title[Weighted mixed endpoint  estimates]{Weighted mixed endpoint  estimates of Fefferman-Stein type for commutators of singular integral operators}
\author{}
\date{}

\usepackage[a4paper, left=2cm, right=2cm, top=3cm, bottom=3cm]{geometry} 

\theoremstyle{plain}
   \newtheorem{teo}{Theorem}
   
   \newtheorem{lema}[teo]{Lemma}

\theoremstyle{definition}
   
\theoremstyle{remark}
 \newtheorem{obs}{Remark}

\numberwithin{equation}{section}
\numberwithin{teo}{section}
\allowdisplaybreaks

\definecolor{aquamarine}{rgb}{0.5, 1.0, 0.83}
\definecolor{americanrose}{rgb}{1.0, 0.01, 0.24}
\definecolor{arsenic}{rgb}{0.23, 0.27, 0.29}
\definecolor{blizzardblue}{rgb}{0.67, 0.9, 0.93}
\definecolor{blush}{rgb}{0.87, 0.36, 0.51}
\definecolor{celestialblue}{rgb}{0.29, 0.59, 0.82}
\definecolor{chocolate(web)}{rgb}{0.82, 0.41, 0.12}
\definecolor{brightpink}{rgb}{1,0,0.5}
\definecolor{cadmiunred}{rgb}{0.89,0,0.13}

\hypersetup{
	colorlinks = true,
	linkcolor = cadmiunred,
	anchorcolor = blue,
	citecolor = brightpink,
	filecolor = blue,
	urlcolor = blue
}

\newcounter{BPR2}

\begin{document}

\author[F. Berra]{Fabio Berra}
\address{Fabio Berra, CONICET and Departamento de Matem\'{a}tica (FIQ-UNL), Santa Fe, Argentina.}
\email{fberra@santafe-conicet.gov.ar}
	
\author[G. Pradolini]{Gladis Pradolini}
\address{Gladis Pradolini, CONICET and Departamento de Matem\'{a}tica (FIQ-UNL), Santa Fe, Argentina.}
\email{gladis.pradolini@gmail.com}
	
\author[J. Recchi]{Jorgelina Recchi}
	\address{Jorgelina Recchi, Departamento de Matemática, Universidad Nacional del Sur (UNS), Instituto de Matematica (INMABB), Universidad Nacional del Sur-CONICET,  Bahía Blanca, Argentina.}
	\email{jrecchi@gmail.com}

 
\thanks{The author were supported by CONICET, UNL, ANPCyT and UNS}



\maketitle

\begin{abstract}	
We deal with mixed weak estimates of Fefferman-Stein type for higher order commutators of Calderón-Zygmund operators with BMO symbol. The results obtained are Fefferman-Stein inequalities that include the estimates proved in \cite{BCP22(JMS)} for the case of singular integral operators, as well as the classical weak endpoint estimate for commutators given in \cite{PP01}.

We also consider commutators of operators involving less regular kernels satisfying an $L^{\Phi}$--Hörmander condition. Particularly, the obtained results contain some previous  estimates proved in \cite{BCP22(JMS)} and \cite{Lorente-Martell-Perez-Riveros}.
\end{abstract}


\section{Introduction}

In 1985, Sawyer established an inequality on the real line for $A_1$ weights. The operator involved was a perturbation of the classical Hardy-Littlewood maximal function $M$ and a motivation to study these estimates was to give an alternative way to prove that $M$ is bounded on $L^p(w)$ where $1<p<\infty$ and $w\in A_p$. Concretely, Sawyer showed that if $u$ and $v$ are $A_1$ weights, then the estimate
\[uv\left(\left\{x\in \mathbb{R}: \frac{M(fv)(x)}{v(x)}>t\right\}\right)\leq \frac{C}{t}\int_{\mathbb{R}}|f|uv\]
holds for some absolute constant $C$ and every $t>0$. 

Although other authors have previously dealt with estimates of this type in the literature (see, for example, \cite{Andersen-Muckenhoupt} or \cite{M-W-77}), the work of Sawyer caught great attention on researchers and many extensions of the estimate above were subsequently established. In \cite{CruzUribe-Martell-Perez} mixed inequalities were obtained for both $M$ and Calderón-Zygmund operators (CZO) in higher dimensions. Aside the condition on the weights considered by Sawyer, the authors proved mixed inequality for related weights $u$ and $v$, in the sense that the product $uv$ belongs to $A_\infty$, allowing to apply classical theory arguments such as the Calderón-Zygmund decomposition with doubling measures.

Many years later, mixed estimates for commutators of Calderón-Zygmund operators were also established in \cite{Berra-Carena-Pradolini(M)}, and in \cite{Berra-Carena-Pradolini(J)} for the fractional setting. Moreover, in \cite{Berra}, \cite{Berra22Pot} and \cite{Berra-Carena-Pradolini(MN)}  the same authors showed  extensions to generalized maximal operators associated to 
 a given Young with certain properties. A more general result than those appearing in \cite{CruzUribe-Martell-Perez} was also established in \cite{L-O-P}.

On the other hand, a well-known weighted estimate for $M$ due to Fefferman and Stein in \cite{FS-71} establishes that
\[\int_{\mathbb{R}^n} |Mf(x)|^pw(x)\,dx\leq C\int_{\mathbb{R}^n} |f(x)|^p Mw(x)\,dx,\]
for every nonnegative locally integrable function $w$ and $1<p<\infty$. These type of estimates play an important role in Harmonic Analysis, particularly when duality arguments are required. Later on, in \cite{CF76} a similar estimate for CZO were established, namely
\[\int_{\mathbb{R}^n} |Tf(x)|^pw(x)\,dx\leq C_{p,r}\int_{\mathbb{R}^n} |f(x)|^p M\circ M_r(w)(x)\,dx,\]
where $1<p,r<\infty$. Some years later the estimate above was improved by Wilson in \cite{Wilson-89}, where the composition $M\circ M_r$ was replaced by the smaller operator $M^2\approx M_{L\text{log}L}$, for the case $1<p<2$.

In \cite{P94London}, Pérez established a generalization of Wilson's estimate for the entire range $1<p<\infty$. Further extensions for commutators of CZO were also established in \cite{Perez_Sharp97}. 
Endpoint weak inequalities of this type were also considered (see \cite{P94London} for CZO and \cite{PP01} for their commutators with BMO symbol).

Inspired by the work in \cite{P94London}, Berra, Carena and Pradolini (\cite{BCP22(JMS)}) proved mixed weighted estimates that generalize Fefferman-Stein type inequalities for CZO. They also consider  
 less regular operators involving a kernel with smoothness conditions given by means of certain Young functions. Although this article brought a first approach to the topic of mixed inequalities of Fefferman-Stein type, it was proved in \cite{Berra-Carena-Pradolini(M)} that the estimate
\[uw\left(\left\{x\in\mathbb{R}^n: \frac{M_{\Phi}(fv)(x)}{v(x)}>t\right\}\right)\leq C\int_{\mathbb{R}^n}\Phi\left(\frac{|f|v}{t}\right)Mu\]
holds, where $M_\Phi$ is a maximal operator related to an $L\text{log} L$ type function (see Section~\ref{seccion: preliminares} for the precise definition), and $w$ depends on $\Phi$ and $v$. 

In this article we are concerned in giving mixed weak estimates of Fefferman-Stein type for higher order commutators of CZO with BMO symbol.
The results obtained are not only interesting by themselves, but also extend the estimates proved in  \cite{BCP22(JMS)} and the weak endpoint inequalities given in \cite{P94London} and \cite{PP01}.

In order to state our main results we give some previous definitions.
Recall that a linear operator $T$ is a CZO if it is bounded on $L^2(\mathbb{R}^n)$ and, for $f\in L^2$ with compact support, we have the representation
	\begin{equation}\label{eq: representacion integral de T}Tf(x)=\int_{\mathbb{R}^n}K(x-y)f(y)\,dy ,\quad\quad x\notin \mathrm{supp}(f).\end{equation}

 The kernel  $K\colon\mathbb{R}^n\backslash\{0\}\to\mathbb{C}$ is a measurable function defined away from the origin that satisfies a size condition given by 
	\[|K(x)|\lesssim \frac{1}{|x|^n},\]
	and the following smoothness condition 
	\begin{equation}\label{eq:prop del nucleo}
	|K(x-y)-K(x-z)|\lesssim \frac{|x-z|}{|x-y|^{n+1}},\quad \textrm{ if } |x-y|>2|y-z|.
	\end{equation}
	The notation $A\lesssim B$ means, as usual, that there exists a positive constant $c$ such that $A\leq cB$. When $A\lesssim B$ and $B\lesssim A$ we shall write $A\approx B$.

Along this article we shall be dealing with the following two functions.  Given $m\in\mathbb{N}$ and $\varepsilon>0$, we define 
\[\Phi_m^\varepsilon(\lambda)=\lambda(1+\log^+\lambda)^{m+\varepsilon},\]
where, as usual, $\log^+\lambda=\max\{0,\log\lambda\}$. For $\varepsilon=0$ we write $\Phi_m^0=\Phi_m$. Additionally, for $p>q>2$, we define 
\[\Psi(\lambda)=\lambda^{p'+1-q'}\mathcal{X}_{[0,1)}(\lambda)+\lambda^{p'}\mathcal{X}_{[1,\infty)}(\lambda).\]
 
We are now in a position to state our first main result.
\begin{teo}\label{teo: mixta tipo F-S para T_b^m}
        Let $u$ be a nonnegative and locally integrable function. Let $q>2$ and $v\in \mathrm{RH}_\infty\cap A_q$. Let $m\in\mathbb{N}$, $b\in\mathrm{BMO}$ and $\varepsilon>0$. If $T$ is a CZO, then for every $p>\max\{q, 1+(m+1)/\varepsilon\}$ the inequality
		\[uv\left(\left\{x\in \mathbb{R}^n: \frac{|T_b^m(fv)(x)|}{v(x)}>t\right\}\right)\leq C\int_{\mathbb{R}^n}\Phi_m\left(\|b\|_{\mathrm{BMO}}^m\frac{|f(x)|}{t}\right)M_{\Phi_{m}^\varepsilon, v^{1-q'}}u(x)M(\Psi\circ v)(x)\,dx\]
		holds for every positive $t$.   
	\end{teo}

 \begin{obs}
 By technical reasons, we shall require that $\alpha=p'+1-q'$ be a positive number. When $q>2$ this is guaranteed since $q'<2$ and therefore $p'+1-q'>p'-1>0$. If $1<q<2$, this condition holds provided $p<1/(2-q)$. This would lead to a restriction for $\varepsilon$ given. Nevertheless, if $v\in A_q$ with $1<q<2$, then $v\in A_r$ for every $r>2$. This means that it will be enough to consider $q>2$ and remove the extra restriction for $p$.
 \end{obs}

 As we have noticed, the theorem above generalizes some previous results known in the literature. For example, $T_b^0=T$ and then we get the mixed estimate of Fefferman-Stein type for CZO obtained in \cite{BCP22(JMS)}. For $m\geq 1$ this estimate extends the endpoint inequality given in \cite{PP01} when $v=1$.

 \medskip

 We do not only consider operators as in \eqref{eq: representacion integral de T} with kernel satisfying condition \eqref{eq:prop del nucleo}, but also kernels with less regular condition. The motivation of dealing with them is that the classical Hörmander condition on $K$ fails to get Coifman type estimates for these operators (see \cite{Lorente-Riveros-delaTorre05} and \cite{Martell-Perez-Trujillo}).

We now introduce some notation that we shall be dealing with. Given a Young function $\varphi$, we denote
\[\|f\|_{\varphi,|x|\sim s}=\left\|f\mathcal{X}_{|x|\sim s}\right\|_{\varphi, B(0,2s)}\]
where $|x|\sim s$ means that $s<|x|\leq 2s$ and $\|\cdot\|_{\varphi,B(0,2s)}$ denotes the Luxemburg average over the ball $B(0,2s)$ (see the next sections for further details). 

We say that a kernel $K$ satisfies an $L^{\varphi}-$Hörmander condition, and we denote it by $K\in H_\varphi$, if there exist constants $c\geq 1$ and $C_\varphi>0$ such that the inequality
\begin{equation}\label{eq: condicion Hormander}
\sum_{k=1}^\infty (2^kR)^n\|K(\cdot-y)-K(\cdot)\|_{\varphi,|x|\sim 2^kR}\leq C_\varphi
\end{equation}
holds for every $y\in\mathbb{R}^n$ and $R>c|y|$. When $\varphi(t)=t^r$, $r\geq 1$, we write $H_\varphi=H_r$. 

We also say that a kernel $K\in H_{\varphi,m}$, $m\in\mathbb{N}$, if there exist two constants $c\geq 1$ and $C_{\varphi,m}>0$ such that
\begin{equation}\label{eq: condicion Hormander - m}
\sum_{k=1}^\infty (2^kR)^nk^m\|K(\cdot-y)-K(\cdot)\|_{\varphi,|x|\sim 2^kR}\leq C_{\varphi,m}
\end{equation}
holds for every $y\in\mathbb{R}^n$ and $R>c|y|$. Observe that $H_{\varphi,m}\subset H_{\varphi,\ell}$, for every $0\leq \ell\leq m$.

Both Fefferman-Stein and  Coifman estimates for operators with kernels of the type defined above were obtained in \cite{Lorente-Martell-Perez-Riveros} and \cite{LRdlT}.

 We say that a Young function $\varphi$ has an upper type $p$, $0<p<\infty$, if there exists a positive constant $C$ such that $\varphi(st)\leq Cs^p\varphi(t)$, for every $s\geq 1$ and $t\geq0$. 
  We also say that $\varphi$ has a lower type $q$ if there exists $C>0$ such that the inequality $\varphi(st)\leq Cs^q\varphi(t)$ holds for every $0\leq s\leq 1$ and $t\geq0$. 

We are now in a position to state our main result involving the kernels $K\in H_\varphi$. For our purposes, we will be assuming that both $\xi$ and $\tilde\xi$ are Young functions.

 \begin{teo}\label{teo: mixta tipo F-S para T_b^m Hormander}
 Let $m\in\mathbb{N}$ and $\xi,\zeta$ be Young functions such that $\tilde\xi$ has both an upper type $r$ and a lower type $s$, with $1<s<r<2$, and $\tilde\xi^{-1}(\lambda)\zeta^{-1}(\lambda)(\log \lambda)^m\lesssim \lambda$, for every $\lambda\geq \lambda_0\geq e$. 
 Let $T$ be defined as in \eqref{eq: representacion integral de T} with a kernel $K\in H_\zeta\cap H_{\xi,m}$. Assume that there exists $r<p<r'$ and Young functions $\eta$ and $\varphi$ such that $\eta\in B_{p'}$ and $\eta^{-1}(\lambda)\varphi^{-1}(\lambda)\lesssim \tilde\xi^{-1}(\lambda)$, for $\lambda\geq \lambda_0$. Given a nonnegative and locally integrable function $u$ and  $v\in \mathrm{RH}_\infty\cap A_q$ with $q=1+(p-1)/r$, if $b\in\mathrm{BMO}$, we have that
		\[uv\left(\left\{x\in \mathbb{R}^n: \frac{|T_b^m(fv)(x)|}{v(x)}>t\right\}\right)\leq C\int_{\mathbb{R}^n}\Phi_m\left(\|b\|_{\mathrm{BMO}}^m\frac{|f(x)|}{t}\right)M_{\varphi_p, v^{1-q'}}u(x)M(\Psi\circ v)(x)\,dx\]
		holds for every positive $t$, where $\varphi_p(\lambda)=\varphi(\lambda^{1/p})$. 
 \end{teo}

 This result can be seen as an extension of the mixed inequality  proved in \cite{BCP22(JMS)}. Furthermore, if we take $v=1$ it also extends a weak endpoint Fefferman-Stein estimate (see Theorem~3.8 in \cite{Lorente-Martell-Perez-Riveros}).

 \medskip

 We shall exhibit an example of Young functions satisfying the hypotheses of the theorem above. Fix $m\in\mathbb{N}$ and $1<r<2$. Let $s$ and $p$ such that $1<s<r<p<r'$. For $0<\varepsilon<\min\{r-s, p'-r\}$ and $0<\delta<m(r-\varepsilon)$ we take 
 \[\tilde\xi(\lambda)=\lambda^{r-\varepsilon}(1+\log^+\lambda)^{\delta} \quad \text{and} \quad \eta(\lambda)=\lambda^{p'-\tau},\]
 where $0<\tau<\min\{p'-1,p'-r-\varepsilon\}$. It is clear from this choice that $\tilde\xi$ has an upper type $r$ and a lower type $s$, and also that $\eta\in B_{p'}$.
 
 If $\beta\geq (1/(r-\varepsilon)-1/(p'-\tau))^{-1}$ and $\alpha\geq \beta\delta/(r-\varepsilon)$,
 we also define the functions
 \[\varphi(\lambda)=\lambda^\beta(1+\log^+\lambda)^\alpha \quad \text{ and } \quad \zeta(\lambda)=\lambda^\theta(1+\log^+\lambda)^\nu,\]
 where $\theta=(r-\varepsilon)'$ and $\displaystyle\nu=\theta\left(m-\frac{\delta}{r-\varepsilon}\right)$. Observe that $\tilde\xi$, $\eta$, $\varphi$ and $\zeta$ are Young functions from our choice of the parameters. Moreover, for $\lambda\geq e$ we have that
 \[\tilde\xi^{-1}(\lambda)\zeta^{-1}(\lambda)(\log\lambda)^m\approx \lambda^{1/(r-\varepsilon)}(\log\lambda)^{-\delta/(r-\varepsilon)}\lambda^{1/\theta}(\log\lambda)^{-\nu/\theta}(\log\lambda)^m=\lambda,\]
 and also
 \[\eta^{-1}(\lambda)\varphi^{-1}(\lambda)\approx \lambda^{1/(p'-\tau)}\lambda^{1/\beta}(\log\lambda)^{-\alpha/\beta}\lesssim \lambda^{1/(r-\varepsilon)}(\log\lambda)^{-\delta/(r-\varepsilon)}\approx \tilde\xi^{-1}(\lambda),\]
 as required. 
 
The remainder of the paper is organized as follows: in Section~\ref{seccion: preliminares} we give some definitions and previous results. In Section~\ref{seccion: auxiliares} we prove some technical and auxiliary results that we need for the proofs of Theorem~\ref{teo: mixta tipo F-S para T_b^m} and~\ref{teo: mixta tipo F-S para T_b^m Hormander}, contained in Section~\ref{seccion: F-S para OCZ} and~\ref{seccion: F-S para Hormander}, respectively.

\medskip
	
	\section{Preliminaries and definitions}\label{seccion: preliminares}

 By the notation $A\lesssim B$ we shall mean that there exists a positive constant $C$ such that $A\leq CB$. We say that $A\approx B$ if $A\lesssim B$ and $B\lesssim A$.
 
 We recall that a weight  $w$ is a locally integrable function that verifies $0<w(x)<\infty$ for almost every $x$. 

 Given $1<p<\infty$, we say that a weight $w$ belongs to the Muckenhoupt $A_p$ class if there exists a positive constant $C$ such that the inequality
 \[\left(\frac{1}{|Q|}\int_Q w\right)\left(\frac{1}{|Q|}\int_Q w^{1-p'}\right)^{p-1}\leq C\]
holds for every cube $Q$ with sides parallel to the coordinate axes and, when $p=1$, $w\in A_1$ if there exists a positive constant $C$ such that
 \[\frac{1}{|Q|}\int_Q w\leq Cw(x)\]
 for almost every $x\in Q$.

 The smallest constant $C$ that can be chosen in the respective inequalities above is denoted by $[w]_{A_p}$. 

 We say that $w\in A_\infty$ if it belongs to $A_p$ for some $1\leq p<\infty$, that is, $A_\infty=\bigcup_{p\geq 1} A_p$. For standard properties of $A_p$ weights see, for example, \cite{javi} and \cite{grafakos}. 
 
 Given $1<s<\infty$, we say that $w$ belongs to the reverse Hölder class $\text{RH}_s$ if there exists a positive constant $C$ such that
\[\left(\frac{1}{|Q|}\int_Q w^s\right)\leq \frac{C}{|Q|}\int_Q w\]
holds for every cube $Q$. It is well-known that if $w\in A_p$, then $w\in \text{RH}_s$ for some $1<s<\infty$.
By $w\in\text{RH}_\infty$ we understand that the following inequality 
\[\sup_Q w\leq \frac{C}{|Q|}\int_Q w\]
holds for some positive constant $C$ and every cube $Q$. The smallest constant for the $\text{RH}_s$ condition to hold is denoted by $[w]_{\text{RH}_s}$, $1<s\leq \infty$. It is easy to check that $\text{RH}_\infty\subset\text{RH}_s\subset \text{RH}_t$, whenever $1<t<s$.
 
The next lemma was proved in \cite{Cruz-Uribe-Neugebauer} and establishes some useful properties of $\rm{RH}_\infty$ classes that we shall use throughout the paper.

	\begin{lema}\label{lema: potencia negativa de RHinf en A1 y positivas en RHinf}
	Let $w$ be a weight.
	\begin{enumerate}[\rm (a)]
		\item \label{item a - lema: potencia negativa de RHinf en A1 y positivas en RHinf}If $p>1$ and $w\in\mathrm{RH}_\infty\cap A_p$, then $w^{1-p'}\in A_1$;
		\item \label{item b - lema: potencia negativa de RHinf en A1 y positivas en RHinf}if $w\in \mathrm{RH}_\infty$, then $w^r\in \mathrm{RH}_\infty$ for every $r>0$;
		\item \label{item c - lema: potencia negativa de RHinf en A1 y positivas en RHinf}if $w\in A_1$, then $w^{-1}\in \rm{RH}_\infty$.
	\end{enumerate}
\end{lema}

Let $\varphi\colon [0,\infty)\to[0,\infty)$ be a Young function, that is,  a strictly increasing and convex function that verifies $\varphi(0)=0$ and $\lim_{t\to\infty} \varphi(t)=\infty$. 

The complementary function $\tilde\varphi$ of such $\varphi$ is given by
\[\tilde\varphi(t)=\sup\{ts-\varphi(s): s\geq 0\}.\]
If $\varphi$ and $\tilde\varphi$ are Young functions, the relation
\begin{equation}\label{eq: producto de inversas como t}
\varphi^{-1}(t)\tilde\varphi^{-1}(t)\approx t
\end{equation} 
holds for every $t$ (see, for example, \cite{KR} or \cite{raoren}).

Given a Young function $\varphi$ and a weight $w$, the generalized maximal function $M_{\varphi,w}$ is defined, for $f$ such that $\varphi(f)\in L^1_{\text{loc}}(w)$, by
\[M_{\varphi, w}f(x)=\sup_{Q\ni x}\|f\|_{\varphi, Q, w},\]
where $\|f\|_{\varphi, Q, w}$ denotes the weighted Luxemburg averages of $f$ over $Q$, given by
\[\|f\|_{\varphi,Q,w}=\inf\left\{\lambda>0: \frac{1}{w(Q)}\int_Q\varphi\left(\frac{|f|}{\lambda}\right)w\leq 1\right\}.\]
In fact, the infimum above is actually a minimum, since it can be seen that
\[\frac{1}{w(Q)}\int_Q\varphi\left(\frac{|f|}{\|f\|_{\varphi,Q,w}}\right)w\leq 1.\]

When $w=1$ we just write $M_{\varphi,w}=M_\varphi$. 
If we further take, for $r\geq 1$, $\varphi(t)=t^r$, then
\[M_\varphi f=M_r f=M(f^r)^{1/r},\]
where $M$ is the classical Hardy-Littlewood maximal operator.

There is a useful relation between weighted Luxemburg averages in terms of modular expressions. A proof for the case $w=1$ can be found in \cite{KR}, although the same proof also works for  doubling measures $\mu$, particularly for $d\mu(x)=w(x)\,dx$ where $w$ is a Muckenhoupt weight. Concretely, if $\varphi$ is a Young function and $w$ is a Muckenhoupt weight, we have that
\begin{equation}\label{eq: equivalencia norma Luxemburgo con infimo}
\|f\|_{\varphi,Q,w}\approx \inf_{\tau>0}\left\{\tau+\frac{\tau}{w(Q)}\int_Q \varphi\left(\frac{|f|}{\tau}\right)w\right\},
\end{equation}
 for every cube $Q$.

 If $\varphi$, $\psi$ and $\eta$ are Young functions that verify
 \[\eta^{-1}(t)\psi^{-1}(t)\lesssim \varphi^{-1}(t)\]
 for every $t\geq t_0>0$, we can conclude that there exists $K_0>0$ such that
 \begin{equation}\label{eq: consecuencia relacion de inversas}
 \varphi(st)\lesssim \eta(s)+\psi(t)
 \end{equation}
 for $s,t\geq K_0$. As a consequence, the generalized version of Hölder inequality for Luxemburg averages
 \begin{equation}\label{eq: Holder generalizada con promedios Luxemburgo}
 \|fg\|_{\varphi, Q, w}\lesssim \|f\|_{\eta, Q, w}\|g\|_{\psi, Q, w}
 \end{equation}
holds for every doubling weight $w$ and every cube $Q$. Particularly, when we take $w=1$, by means of \eqref{eq: producto de inversas como t} we have that
\begin{equation}\label{eq: desigualdad de Holder generalizada}
\frac{1}{|Q|}\int_Q |fg|\lesssim \|f\|_{\varphi, Q}\,\|g\|_{\tilde\varphi, Q}.
\end{equation}

Given $1<p<\infty$, we say that $\varphi\colon[0,\infty)\to [0,\infty)$ belongs to $B_p$ if there exists a constant $c>0$ such that
\[\int_c^\infty \frac{\varphi(t)}{t^p}\frac{\,dt}{t}<\infty.\]
These classes were introduced in \cite{Perez-95-Onsuf} and are related with the boundedness of the operator $M_\varphi$ in $L^p$. It is not difficult to see that if $\varphi\in B_p$, then $\varphi(t)\lesssim t^p$.
 
The following lemmas will be useful in the sequel. A proof  can be found in \cite{BCP22(JMS)}.

\begin{lema}\label{lema: comparacion peso por caracteristica}
	Let $\varphi$ be a Young function, $w$ a doubling weight, $f$ a function such that $M_{\varphi,w} f (x)<\infty$ a. e. and $Q$ be a fixed cube. Then
	\[M_{\varphi,w}(f\mathcal{X}_{\mathbb{R}^n\backslash RQ})(x)\approx M_{\varphi,w}(f\mathcal{X}_{\mathbb{R}^n\backslash RQ})(y)\]
	for every $x,y\in Q$, where $R=4\sqrt{n}$. 
\end{lema}

\begin{lema}\label{lema: relacion maximal generalizada con peso A1}
	Let $w\in A_1$ and $\varphi$ be a Young function.
	\begin{enumerate}[\rm (a)]
		\item \label{item a - lema: relacion maximal generalizada con peso A1}  There exists a positive constant $C$ such that
  \[\|f\|_{\varphi,Q} \leq C\|f\|_{\varphi,Q,w}\]
  for every cube $Q$. As a consequence, $M_{\varphi}f(x)\leq C M_{\varphi,w}f(x)$,
		for every function $f$ such that $M_{\varphi,w}f<\infty$ a.e.;
		\item \label{item b - lema: relacion maximal generalizada con peso A1} if $w^r\in A_1$ for some $r>1$, then
		\[M_{\varphi,w}f(x)\leq C M_{\varphi,w^r}f(x),\]
		for every function $f$ such that $M_{\varphi,w^r}f<\infty$ a.e.
		\end{enumerate}
\end{lema}

The next lemma establishes a well-known bound for functions of $L\log L$ type that we shall require in our main estimates. A proof can be found, for example, in \cite{Berra}.

\begin{lema}\label{lema: acotacion LlogL por potencia}
Let $\delta>0$ and $\varphi(t)=t(1+\log^+t)^\delta$. For every $\varepsilon>0$ there exists a positive constant $C=C(\varepsilon,\delta)$ such that
\[\varphi(t)\leq Ct^{1+\varepsilon}, \quad\textrm{ for }\quad t\geq 1.\]
Moreover, the constant $C$ can be taken as $C=\max\left\{1, (\delta/\varepsilon)^\delta\right\}.$	
\end{lema}

Given a linear operator $T$ and a locally integrable function $b$, the first order commutator operator of $T$ is formally defined by
\[T_bf=[b,T]f=bTf-T(bf).\]
We can also consider higher order commutators proceeding recursively. For $m\in \mathbb{N}$, we define
\[T_b^mf=[b,T_b^{m-1}]f,\]
where we understand $T_b^0=T$.

If the operator $T$ has an integral representation as \eqref{eq: representacion integral de T}, it is easy to check that the corresponding higher order commutator verifies
\begin{equation}
T_b^mf(x)=\int_{\mathbb{R}^n}(b(x)-b(y))^mK(x-y)f(y)\,dy
\end{equation}
for every $x\not\in \text{supp}(f)$.

The function $b$ is usually known as the symbol of the commutator $T_b^m$ and in this article we shall deal with symbols 
belonging to the bounded mean oscillation space \text{BMO}, given by the locally integrable functions $b$ such that
\[\|b\|_{\text{BMO}}=\sup_Q \frac{1}{|Q|}\int_Q|b-b_Q|<\infty,\]
where $b_Q$ stands for the usual average of $b$ over the cube $Q$. The quantity $\|b\|_{\text{BMO}}$ does not correspond to a norm, but it does if we restrict it to the quotient of the $\text{BMO}$ space modulo the  constant functions.

The following classical results are well-know estimates for BMO functions (see, for example, \cite{Perez95}). 

\begin{lema}\label{lema: diferencia de promedios en cubos dilatados acotados por norma BMO}
Let $b\in \mathrm{BMO}$. There exists a positive constant $C$ such that for every $k\in\mathbb{N}$ and every cube $Q$ the inequality
\[|b_Q-b_{2^kQ}|\leq C k\|b\|_{\mathrm{BMO}}\]
holds, where $2^kQ$ is the cube with the same center as $Q$ and side length $2^k$ times the side length of $Q$.
\end{lema}

\begin{lema}\label{lema: norma exponencial de la oscilacion media acotada por la norma BMO}
Let $b\in \mathrm{BMO}$, $\delta>0$ and $\psi(t)=e^{t^{1/\delta}}-1$. Then there exists a positive constant $C$ such that the inequality
\[\|b-b_Q\|_{\psi, Q}\leq C \|b\|_{\mathrm{BMO}}\]
holds for every cube $Q$.
\end{lema}

	\section{Auxiliary results}\label{seccion: auxiliares}
We devote this section to state and prove some results that will be useful for our purposes. The first two lemmas contain estimates previously proved in \cite{BCP22(JMS)}. We include both proofs for the sake of completeness. Recall that
$\Phi_m^\varepsilon(\lambda)=\lambda(1+\log^+\lambda)^{m+\varepsilon}$ and $\Psi(\lambda)=\lambda^{p'+1-q'}\mathcal{X}_{[0,1)}(\lambda)+\lambda^{p'}\mathcal{X}_{[1,\infty)}(\lambda)$.

\begin{lema}\label{lema: estimacion de maximal del producto de pesos}
Let $m\in\mathbb{N}$, $\varepsilon>0$ and $1<q<p$. Let $u$ be a nonnegative and locally integrable function and $v\in {\rm RH}_\infty\cap A_q$. Then there exists a positive constant $C$ such that the inequality
\[M_{\Phi_m^\varepsilon}\left(uv^{1-p'}\right)(x)\leq CM_{\Phi_m^\varepsilon,v^{1-q'}}u(x)\,v^{-p'}(x)(\Psi\circ v)(x)\]
holds for almost every $x$.
\end{lema}

\begin{proof}
 The hypothesis imply that $v\in \mathrm{RH}_\infty\cap A_p$ since $p>q$, so by item~\eqref{item a - lema: potencia negativa de RHinf en A1 y positivas en RHinf} of Lemma~\ref{lema: potencia negativa de RHinf en A1 y positivas en RHinf} we get $v^{1-p'}\in A_1$. 
	
	Fix $x$ and $Q$ a cube containing $x$. We pick $\lambda=\|u\|_{\Phi_m^\varepsilon, Q, v^{1-q'}}$, and write
	\begin{align*}
	\frac{1}{|Q|}\int_Q \Phi_m^\varepsilon\left(\frac{uv^{1-p'}}{\lambda}\right)&=\frac{1}{|Q|}\int_{Q\cap\{v^{1-p'}\leq e\}} \Phi_m^\varepsilon\left(\frac{uv^{1-p'}}{\lambda}\right)+\frac{1}{|Q|}\int_{Q\cap\{v^{1-p'}>e\}} \Phi_m^\varepsilon\left(\frac{uv^{1-p'}}{\lambda}\right)\\
	&=I_1+I_2.
	\end{align*}
	By using that $\Phi_m^\varepsilon$ is submultiplicative and Lemma~\ref{lema: relacion maximal generalizada con peso A1},  for $I_1$ we get
	\[I_1\lesssim \frac{1}{|Q|}\int_Q\Phi_m^\varepsilon\left(\frac{u}{\|u\|_{\Phi_m^\varepsilon, Q , v^{1-q'}}}\right)\lesssim \frac{1}{|Q|}\int_Q\Phi_m^\varepsilon\left(\frac{u}{\|u\|_{\Phi_m^\varepsilon, Q}}\right)\leq 1.\]
	 To deal with $I_2$, let $\tau=(q'-1)/(p'-1)-1>0$, since $p'<q'$. By applying Lemma~\ref{lema: acotacion LlogL por potencia}  we get that  
	 \begin{align*}
	 	I_2&\lesssim \frac{1}{|Q|}\int_Q \Phi_m^\varepsilon\left(\frac{u(y)}{\lambda}\right)v^{(1-p')(1+\tau)}\,dy\\
	&= \frac{1}{|Q|}\int_Q \Phi_m^\varepsilon\left(\frac{u(y)}{\lambda}\right)v^{1-q'}(y)\,dy\\
	&=\frac{v^{1-q'}(Q)}{|Q|}\left(\frac{1}{v^{1-q'}(Q)}\int_Q\Phi_m^\varepsilon\left(\frac{u(y)}{\|u\|_{\Phi_m^\varepsilon,Q,v^{1-q'}}}\right)v^{1-q'}(y)\,dy\right)\\
	&\lesssim \left[v^{1-q'}\right]_{A_1}\max\left\{1,v^{1-q'}(x)\right\},
	 \end{align*}
	 since $v^{1-q'}\in A_1$. Consequently, we arrive to
	 \begin{align*}
	 \|uv^{1-p'}\|_{\Phi_m^\varepsilon,Q}&\lesssim \max\left\{1,v^{1-q'}(x)\right\}\|u\|_{\Phi_m^\varepsilon, Q, v^{1-q'}}\\
      &\leq  \max\left\{1,v^{1-q'}(x)\right\} M_{\Phi_m^\varepsilon,v^{1-q'}}u(x).
	 \end{align*}
  By noticing that
  \[\max\{1,v^{1-q'}(x)\}= v^{-p'}(x)(\Psi\circ v)(x)\]
  and taking supremum over the cubes $Q$ that contain $x$ we achieve the desired inequality. 
\end{proof}

\begin{lema}\label{lema: estimacion de maximal del producto de pesos - 2}
Let $m\in\mathbb{N}$ and $q>1$. Let $\xi$, $\eta$ and $\varphi$ be Young functions such that $\eta^{-1}(t)\varphi^{-1}(t)\lesssim \tilde\xi^{-1}(t)$, for $t\geq t_0\geq e$. Assume that $\xi$ has an upper type $r>1$ and $\eta\in B_{p'}$, where $p=1+r(q-1)$. Let $u$ be a nonnegative and locally integrable function and $v\in {\rm RH}_\infty\cap A_q$. Then there exists a positive constant $C$ such that the inequality
\[M_{\tilde\xi}\left(uv^{1-p'}\right)(x)\leq CM_{\varphi_p,v^{1-q'}}u(x)\,v^{-p'}(x)(\Psi\circ v)(x)\]
holds for almost every $x$, where $\varphi_p(t)=\varphi(t^{1/p})$.
\end{lema}

\begin{proof}
 Since $p>q$, by item~\eqref{item a - lema: potencia negativa de RHinf en A1 y positivas en RHinf} of Lemma~\ref{lema: potencia negativa de RHinf en A1 y positivas en RHinf}, we have again $v^{1-p'}\in A_1$. Fix $x$ and a cube $Q$ containing $x$. We define $\lambda=\|u\|_{\varphi_p, Q, v^{1-q'}}$, and we split as before 
	\[\frac{1}{|Q|}\int_Q \tilde\xi\left(\frac{uv^{1-p'}}{\lambda}\right)=\frac{1}{|Q|}\int_{Q\cap\{v^{1-p'}\leq 1\}}+\frac{1}{|Q|}\int_{Q\cap\{v^{1-p'}> 1\}}=I_1+I_2.\]
 Since $\eta\in B_{p'}$, we get that $\eta(t)\lesssim t^{p'}$. 
 Therefore, since $\tilde\xi$ is a Young function we obtain that
 \[t\gtrsim \tilde\xi^{-1}(t)\gtrsim \eta^{-1}(t)\varphi^{-1}(t)\gtrsim t^{1/p'}\varphi^{-1}(t)\gtrsim (\varphi^{-1}(t))^{1+p/p'}=\varphi_p^{-1}(t).\]
 This estimate finally yields $\tilde\xi(t)\lesssim \varphi_p(t)$, for every $t\geq t_0$. As a consequence, $I_1$ is bounded by an absolute constant since $\lambda\geq \|u\|_{\varphi_p,Q}$ by virtue of item~\eqref{item a - lema: relacion maximal generalizada con peso A1} of Lemma~\ref{lema: relacion maximal generalizada con peso A1}.

 In order to deal with $I_2$, we use the upper type $r$ of $\tilde\xi$ combined with \eqref{eq: consecuencia relacion de inversas} to obtain
	\begin{align*}
	I_2&\lesssim \frac{1}{|Q|}\int_{Q\cap\{v^{1-p'}>1\}} \tilde\xi\left(\frac{u}{\lambda}\right)v^{r(1-p')}\\
	&=\frac{1}{|Q|}\int_Q \tilde\xi\left(\left(\frac{u}{\lambda}\right)^{1/p}\left(\frac{u}{\lambda}\right)^{1/p'}\right)v^{1-q'}\\
	&\lesssim \frac{1}{|Q|}\int_Q\varphi_p\left(\frac{u}{\lambda}\right)v^{1-q'}+\frac{1}{|Q|}\int_Q\frac{u}{\lambda}v^{1-q'}\\
	&\lesssim \frac{1}{|Q|}\int_Q\varphi_p\left(\frac{u}{\|u\|_{\varphi_p,Q,v^{1-q'}}}\right)v^{1-q'}+\left(\frac{1}{|Q|}\int_Q\frac{u}{\|u\|_{\varphi_p,Q,v^{1-q'}}}v^{1-q'}\right)\\
	&\lesssim \frac{v^{1-q'}(Q)}{|Q|}\\
	&\lesssim\left[v^{1-q'}\right]_{A_1} v^{1-q'}(x),
	\end{align*}
 where we have also used $\eta(t)\lesssim t^{p'}$ and $t\lesssim \varphi_p(t)$.
	 From both estimates we arrive to
  \[\left\|uv^{1-p'}\right\|_{\tilde\xi,Q}\lesssim \max\left\{1,v^{1-q'}(x)\right\}M_{\varphi_p, v^{1-q'}}u(x).\]
  We can conclude the thesis now by proceeding as in the previous lemma.
\end{proof}

The following lemma will be a key on our main proofs.

\begin{lema}\label{lema: M(Psi(v)) esencialmente constante sobre cubos}
Let $1<q<p$ such that $p'+1-q'>0$, $v\in \mathrm{RH}_\infty$ and $\Psi(t)=t^{p'+1-q'}\mathcal{X}_{(0,1)}(t)+t^{p'}\mathcal{X}_{(1,\infty)}(t)$. There exists a positive constant $C$ such that the inequality
\[M(\Psi\circ v)(x)\leq C\inf_Q M(\Psi\circ v)\]
holds for every cube $Q$ and almost every $x\in Q$. 
\end{lema}

\begin{proof}
Let $\alpha=p'+1-q'$. We have $0<\alpha<1$ since $q<p$. From the fact that $t^{p'}< t^\alpha$ for $0<t<1$, it is not difficult to see that
\[\Psi(t_1+t_2)\leq 2^{p'}(\Psi(t_1)+\Psi(t_2)),\]
for every $t_1,t_2\geq 0$.

Fix a cube $Q$ and $x\in Q$. We write $v=v\mathcal{X}_{\mathbb{R}^n\backslash Q^*}+v\mathcal{X}_{Q^*}=v_1+v_2$, where $Q^*=4\sqrt{n} Q$. Therefore
\[M(\Psi(v))(x)\leq 2^{p'}\left(M(\Psi(v_1))(x)+M(\Psi(v_2))(x)\right) .\]
By virtue of Lemma~\ref{lema: comparacion peso por caracteristica} applied with $\varphi(t)=t$ and $w=1$, we have that
\begin{align*}
M(\Psi(v_1))(x)&= M\left(\Psi(v)\mathcal{X}_{\mathbb{R}^n\backslash Q^*}\right)(x)\\
&\lesssim \inf_{Q} M\left(\Psi(v)\mathcal{X}_{\mathbb{R}^n\backslash Q^*}\right)\\
&\leq \inf_{Q} M(\Psi(v)).
\end{align*}
On the other hand, we shall estimate $M(\Psi(v_2))(x)$. If $R$ is any cube that contains $x$, we have 
\[\frac{1}{|R|}\int_R \Psi(v_2)=\frac{1}{|R|}\int_R \Psi(v)\mathcal{X}_{Q^*}\leq \Psi\left(\sup_{Q^*}v\right).\]
 We shall first assume that $0< \sup_{Q^*} v\leq 1$. In this case we have that $v(z)\leq 1$ for almost every $z$ in $Q^*$.  Therefore, since $v\in\mathrm{RH}_\infty$, by Lemma~\ref{lema: potencia negativa de RHinf en A1 y positivas en RHinf} we obtain
	\[\Psi\left(\sup_{Q^*} v\right)=\left(\sup_{Q^*} v\right)^{\alpha}=\sup_{Q^*} v^{\alpha}\leq \frac{\left[v^{\alpha}\right]_{\mathrm{RH_\infty}}}{|Q^*|}\int_{Q^*} v^{\alpha}=\frac{\left[v^{\alpha}\right]_{\mathrm{RH_\infty}}}{|Q^*|}\int_{Q^*} \Psi(v)\leq \left[v^{\alpha}\right]_{\mathrm{RH_\infty}}
	\inf_{Q}M(\Psi(v)).\]
	Now suppose that $\sup_{Q^*} v>1$. Let $Q_1=Q^*\cap\{v\leq 1\}$ and $Q_2=Q^*\cap\{v>1\}$. By proceeding as above, since $0<\alpha<1$, we have that
	\begin{align*}
	    \Psi\left(\sup_{Q^*} v\right)=\sup_{Q^*}v^{p'}\leq \frac{\left[v^{p'}\right]_{\mathrm{RH_\infty}}}{|Q^*|}\int_{Q^*} v^{p'}
	    &\leq \frac{\left[v^{p'}\right]_{\mathrm{RH_\infty}}}{|Q^*|}\left[\int_{Q_1} v^{\alpha}+\int_{Q_2} v^{p'}\right]\\
	    &= \frac{\left[v^{p'}\right]_{\mathrm{RH_\infty}}}{|Q^*|}\int_{Q^*} \Psi(v)\\
	    &\leq \left[v^{p'}\right]_{\mathrm{RH_\infty}}
	\inf_{Q }M(\Psi(v)).
	\end{align*}
 By taking supremum on $Q$, we get that
 \[M(\Psi(v_2))(x)\lesssim \inf_{Q}M(\Psi(v)).\]
 By combining the estimates for $M(\Psi(v_1))$ and $M(\Psi(v_2))$ we arrive to the desired inequality.
\end{proof}

The following technical estimate provides a way to deal with commutators of order $m$ by means of less order commutators.

\begin{lema}\label{lema: T_b^m en terminos de conmutadores de orden menor}
Let $b\in L^1_{\mathrm{loc}}$ and $T$ be a linear operator. Then, for every $\lambda\in \mathbb{R}$ and for almost every $x$ we have that
\[T_b^mf(x)=(b(x)-\lambda)^mTf(x)-T((b-\lambda)^mf)(x)-\sum_{k=1}^{m-1}C_{m,k}T_b^k((b-\lambda)^{m-k}f)(x),\]
where $C_{m,k}=m!/((m-k)!k!)$.
\end{lema}

\begin{proof}
We proceed by induction on $m$. Fix $\lambda$ and observe that the result is immediate for $m=1$, since the last sum is empty. We assume that the equality holds for $m>1$ and we shall prove it for $m+1$. We have that
\begin{align*}
    T_b^{m+1}f(x)&=[b,T_b^m]f(x)=(b(x)-\lambda)T_b^mf(x)-T_b^m((b-\lambda)f)(x)\\
    &=(b(x)-\lambda)^{m+1}Tf(x)-(b(x)-\lambda)T((b-\lambda)^mf)(x)\\
    &\qquad -(b(x)-\lambda)\sum_{k=1}^{m-1}C_{m,k}T_b^k((b-\lambda)^{m-k}f)(x)-(b(x)-\lambda)^mT((b-\lambda)f)(x)\\
    &\qquad+T((b-\lambda)^{m+1}f)(x)+\sum_{k=1}^{m-1}C_{m,k}T_b^k((b-\lambda)^{m+1-k}f)(x)\\
    &=(b(x)-\lambda)^{m+1}Tf(x)-(b(x)-\lambda)T((b-\lambda)^mf)(x)\\
    &\qquad -\sum_{k=2}^{m}C_{m,k-1}T_b^k((b-\lambda)^{m+1-k}f)(x)-(b(x)-\lambda)^mT((b-\lambda)f)(x)+T((b-\lambda)^{m+1}f)(x)\\
    &=(b(x)-\lambda)^{m+1}Tf(x)-(b(x)-\lambda)^mT((b-\lambda)f)(x)-\sum_{k=1}^{m}C_{m,k-1}T_b^k((b-\lambda)^{m+1-k}f)(x)\\
    &=(b(x)-\lambda)^{m+1}Tf(x)-(b(x)-\lambda)^mT((b-\lambda)f)(x)-\sum_{k=1}^{m}C_{m+1,k}T_b^k((b-\lambda)^{m+1-k}f)(x)\\
    &\qquad +\sum_{k=1}^m C_{m,k}T_b^k((b-\lambda)^{m+1-k}f)(x)\\
    &=(b(x)-\lambda)^{m+1}Tf(x)-T((b-\lambda)^{m+1}f)(x)-\sum_{k=1}^{m}C_{m+1,k}T_b^k((b-\lambda)^{m+1-k}f)(x),
\end{align*}
which yields the desired estimate. Notice that we have used the inductive hypothesis twice and the fact that $C_{m,k}+C_{m,k-1}=C_{m+1,k}$. The proof is complete.
\end{proof}

	The following theorem establishes a strong $(p,p)$ Fefferman-Stein estimate for higher order commutators and a proof can be found in \cite{Perez_Sharp97}.
	\begin{teo}\label{teo: F-S tipo fuerte para conmutador de orden m}
	Let $1<p<\infty$, $m\in\mathbb{N}$, $\varepsilon>0$ and $\varphi_\varepsilon(\lambda)=\lambda(1+\log^+\lambda)^{(m+1)p-1+\varepsilon}$. Let $T$ be a Calderón-Zygmund operator and $b\in\mathrm{BMO}$. Then there exists a positive constant $C$ such that the inequality
	\[\int_{\mathbb{R}^n}|T_b^mf(x)|^pw(x)\,dx\leq C\|b\|_{\mathrm{BMO}}^{mp}\int_{\mathbb{R}^n}|f(x)|^pM_{\varphi_\varepsilon}w(x)\,dx\]
	holds for every nonnegative and locally integrable function $w$.
	\end{teo}

 The next theorem was proved in \cite{BCP22(JMS)} and gives a mixed inequality of Fefferman-Stein type for Calderón-Zygmund operators. 
	
	\begin{teo}\label{teo: mixta tipo F-S para T}
        Let $u$ be a nonnegative and locally integrable function, $q>1$ and $v\in \mathrm{RH}_\infty\cap A_q$. Let $\delta>0$ and $T$ be a CZO. If $\varphi(\lambda)=\lambda(1+\log^+\lambda)^{\delta}$, then for every $p>\max\{q, 1+1/\delta\}$ the inequality
		\[uv\left(\left\{x\in \mathbb{R}^n: \frac{|T(fv)(x)|}{v(x)}>t\right\}\right)\leq \frac{C}{t}\int_{\mathbb{R}^n}{|f(x)|}M_{\varphi, v^{1-q'}}u(x)M(\Psi\circ v)(x)\,dx\]
		holds for every positive $t$.
	\end{teo}

 Concerning mixed estimates of Fefferman-Stein type for operators associated to kernels with less regularity, the following theorem was also established in \cite{BCP22(JMS)}. Due to the nature of these kernels, additional conditions are required on the Young functions involved.

 \begin{teo}\label{teo: F-S para T Hormander}
	Let $\xi$ be a Young function such that $\tilde\xi$ has an upper type $r$ and a lower type $s$, for some $1<s<r$. Let $T$ be an operator as in \eqref{eq: representacion integral de T}, with kernel $K\in H_{\xi}$. Assume that there exist $1<p<r'$ and Young functions $\eta,\varphi$ such that $\eta\in B_{p'}$ and $\eta^{-1}(\lambda)\varphi^{-1}(\lambda)\lesssim \tilde\xi^{-1}(\lambda)$, for every $\lambda\geq \lambda_0$. If $u$ is a nonnegative and locally integrable function and $v\in \mathrm{RH}_\infty\cap A_q$ with $q=1+(p-1)/r$ then the inequality
	\[uv\left(\left\{x\in \mathbb{R}^n: \frac{|T(fv)(x)|}{v(x)}>t\right\}\right)\leq \frac{C}{t}\int_{\mathbb{R}^n}|f(x)|M_{\varphi_p, v^{1-q'}}u(x)M(\Psi(v))(x)\,dx\]
	holds for every $t>0$, where $\varphi_p(\lambda)=\varphi(\lambda^{1/p})$.
\end{teo}

	\section{Proof of Theorem~\ref{teo: mixta tipo F-S para T_b^m}}\label{seccion: F-S para OCZ}
	We devote this section to prove Theorem~\ref{teo: mixta tipo F-S para T_b^m}. Since we proceed by induction, we split the proof in two parts.
	\begin{proof}[Proof of Theorem~\ref{teo: mixta tipo F-S para T_b^m}, case $m=1$]
	 It will be enough to assume that $u$ is bounded and that $f$ is a positive bounded function with compact support. We can also assume without loss of generality that $\|b\|_{\mathrm{BMO}}=1$. Fixed $t>0$, we  perform the Calder\'on-Zygmund decomposition of $f$ at level $t$ with respect to the measure $d\mu(x)=v(x)\,dx$, which is doubling since $v$ is an $A_\infty$ weight. We obtain a collection of disjoint dyadic cubes $\{Q_j\}_{j=1}^\infty$ satisfying $t<f_{Q_j}^v\leq Ct$, where $f_{Q_j}^v$ stands for the weighted average
	\[\frac{1}{v(Q_j)}\int_{Q_j}f(y)v(y)\,dy.\]
	If $\Omega=\bigcup_{j=1}^{\infty}Q_j$, then $f(x)\leq t$ for almost every $x\in{\mathbb{R}}^n\backslash\Omega$. We also write $f=g+h$, where
	\begin{equation*}
	g(x)=\left\{
	\begin{array}{ccl}
	f(x),& \textrm{ if } &x\in\mathbb{R}^n\backslash\Omega;\\
	f_{Q_j}^v,&\textrm{ if }& x\in Q_j,
	\end{array}
	\right.
	\end{equation*}
	and $h(x)=\sum_{j=0}^{\infty}{h_j(x)}$, with
	\begin{equation*}
	h_j(x)=\left(f(x)-f_{Q_j}^v\right)\mathcal{X}_{Q_j}(x).
	\end{equation*}
	From these definitions we get that $g(x)\leq Ct$ almost everywhere, every
	$h_j$ is supported on $Q_j$ and
	\begin{equation}\label{eq: hj integra cero contra v}
	\int_{Q_j}h_j(y)v(y)\,dy=0.
	\end{equation}
	Let $Q_j^*=4\sqrt{n}Q_j$, and denote $\Omega^*=\bigcup_j Q_j^*$. We write
	\begin{align*}
	uv\left(\left\{x\in \mathbb{R}^n: \left|\frac{T_b(fv)}{v}\right|>t\right\}\right)&\leq uv\left(\left\{x\in \mathbb{R}^n\backslash \Omega^*: \left|\frac{T_b(gv)}{v}\right|>\frac{t}{2}\right\}\right) + uv(\Omega^*)\\
	&+uv\left(\left\{x\in \mathbb{R}^n\backslash \Omega^*: \left|\frac{T_b(hv)}{v}\right|>\frac{t}{2}\right\}\right)\\
	&= I_1+I_2+I_3.
	\end{align*}
	Let us estimate every term above separately. Fix $\varepsilon>0$, $p>\max\{q, 1+2/\varepsilon\}$ and set $u^*=u\mathcal{X}_{\mathbb{R}^n\backslash \Omega^*}$.  For $I_1$, we combine Tchebycheff inequality with Theorem~\ref{teo: F-S tipo fuerte para conmutador de orden m} applied with $p'$, $2(1-p')+\varepsilon>0$ and $w=u^*v^{1-p'}$ in order to get
\[I_1\leq\frac{C}{t^{p'}}\int_{\mathbb{R}^n} |T_b(gv)|^{p'}uv^{1-p'}\mathcal{X}_{\mathbb{R}^n\backslash \Omega^*}
	=\frac{C}{t^{p'}}\int_{\mathbb{R}^n} |T_b(gv)|^{p'}u^*v^{1-p'}
	\leq\frac{C}{t^{p'}}\int_{\mathbb{R}^n} (gv)^{p'}M_{\Phi_1^\varepsilon}\left(u^*v^{1-p'}\right).\]
By applying Lemma~\ref{lema: estimacion de maximal del producto de pesos} we have that
	\begin{equation}\label{eq: teo: mixta tipo F-S para T_b^m - eq1}
	M_{\Phi_1^\varepsilon}\left(u^*v^{1-p'}\right)(x)\leq CM_{\Phi_1^\varepsilon,v^{1-q'}}(u^*)(x)\,v^{-p'}(x)\Psi(v(x)) 
	\end{equation}
	holds for almost every $x$.
	 
	Consequently, we have that 
	\begin{align*}
	I_1&\leq \frac{C}{t^{p'}}\int_{\mathbb{R}^n}(gv)^{p'}\left(M_{\Phi_1^\varepsilon, v^{1-q'}}u^*\right)v^{-p'}\Psi(v)\\
	&\leq \frac{C}{t}\int_{\mathbb{R}^n}g\left(M_{\Phi_1^\varepsilon, v^{1-q'}}u^*\right)\Psi(v)\\
	&=\frac{C}{t}\int_{\mathbb{R}^n\backslash \Omega}f\left(M_{\Phi_1^\varepsilon, v^{1-q'}}u\right)\Psi(v)+\frac{C}{t}\int_{\Omega}f_{Q_j}^v\left(M_{\Phi_1^\varepsilon, v^{1-q'}}u^*\right)\Psi(v).
	\end{align*}
	Let $u_j^*=u\mathcal{X}_{\mathbb{R}^n\backslash Q_j^*}$. By applying Lemma~\ref{lema: comparacion peso por caracteristica} we get \refstepcounter{BPR2}\label{pag: estimacion de I_1}
	\begin{align*} 
	\frac{C}{t}\int_{\Omega}f_{Q_j}^v\left(M_{\Phi_1^\varepsilon, v^{1-q'}}u^*\right)\Psi(v)&\leq\frac{C}{t}\sum_j\int_{Q_j}f_{Q_j}^v\left(M_{\Phi_1^\varepsilon, v^{1-q'}}u_j^*\right)\Psi(v)\\
	&\leq \frac{C}{t}\sum_j\inf_{Q_j} M_{\Phi_1^\varepsilon, v^{1-q'}}u_j^*\frac{(\Psi\circ v)(Q_j)}{v(Q_j)}\int_{Q_j}fv\\
	&\leq \frac{C}{t}[v]_{\mathrm{RH}_\infty}\sum_j\inf_{Q_j} M_{\Phi_1^\varepsilon, v^{1-q'}}u_j^*\frac{(\Psi\circ v)(Q_j)}{|Q_j|}\int_{Q_j}f\\
	&\leq \frac{C}{t}[v]_{\mathrm{RH}_\infty}\sum_j\int_{Q_j}f\left(M_{\Phi_1^\varepsilon, v^{1-q'}}u\right)M(\Psi(v))\\
	&\leq \frac{C}{t}[v]_{\mathrm{RH}_\infty}\int_{\Omega}f\left(M_{\Phi_1^\varepsilon, v^{1-q'}}u\right)\,M(\Psi(v)).
	\end{align*}
	
	For $I_2$, we shall use  Lemma~\ref{lema: potencia negativa de RHinf en A1 y positivas en RHinf} and the fact that $v$ is doubling to obtain \refstepcounter{BPR2}\label{pag: estimacion de I_2}
	\begin{align*}
	uv(Q_j^*)&\leq v^{1-q'}(Q_j^*)\|u\|_{\Phi_1^\varepsilon,Q_j^*,v^{1-q'}}\left[\frac{1}{v^{1-q'}(Q_j^*)}\int_{Q_j^*}\Phi_1^\varepsilon\left(\frac{u}{\|u\|_{\Phi_1^\varepsilon,Q_j^*,v^{1-q'}}}\right)v^{1-q'}\right]\left(\sup_{Q_j^*} v^{q'}\right)\\
	&\leq \left[v^{q'}\right]_{\rm{RH}_\infty}\frac{v^{1-q'}(Q_j^*)}{|Q_j^*|}v^{q'}(Q_j^*)\|u\|_{\Phi_1^\varepsilon,Q_j^*,v^{1-q'}}\\
	&\leq C\left[v^{q'}\right]_{\rm{RH}_\infty}\left[v^{1-q'}\right]_{A_1}v(Q_j)\|u\|_{\Phi_1^\varepsilon,Q_j^*,v^{1-q'}}\\
	&\leq \frac{C}{t}\int_{Q_j}fv\left(M_{\Phi_1^\varepsilon,v^{1-q'}}u\right)\\
	&\leq \frac{C}{t}\int_{Q_j}f\left(M_{\Phi_1^\varepsilon,v^{1-q'}}u\right)M(\Psi(v)),
	\end{align*}
	since we have $\Psi(s)\geq s$.
	Consequently,
	\begin{align*}
	uv(\Omega^*)&\leq \sum_j uv(Q_j^*)\\
	&\leq \frac{C}{t}\sum_j \int_{Q_j}f\left(M_{\Phi_1^\varepsilon,v^{1-q'}}u\right)M(\Psi(v))\\
	& \leq\frac{C}{t} \int_{\mathbb{R}^n}f\left(M_{\Phi_1^\varepsilon,v^{1-q'}}u\right)M(\Psi(v)).
	\end{align*}
		
	It only remains to estimate $I_3$. We have that
	\[T_b(hv)(x)=\sum_j T_b(h_jv)(x)=\sum_j (b(x)-b_{Q_j})T(h_jv)(x)-\sum_jT((b-b_{Q_j})h_jv),\]
	and therefore
	\begin{align*}
	I_3&\leq uv\left(\left\{x\in \mathbb{R}^n\backslash \Omega^*: \left|\sum_j\frac{(b-b_{Q_j})T(h_jv)}{v}\right|>\frac{t}{4}\right\}\right)\\
	&\qquad +uv\left(\left\{x\in \mathbb{R}^n\backslash \Omega^*: \left|\sum_j\frac{T((b-b_{Q_j})h_jv)}{v}\right|>\frac{t}{4}\right\}\right)\\
	&=I_3^1+I_3^2.
	\end{align*}
	
	Let us also denote $A_{j,k}=\{x: 2^{k-1}r_j<|x-x_{Q_j}|\leq 2^{k}r_j\}$, where $r_j=2\sqrt{n}\ell(Q_j)$. By using the integral representation of $T$ given by \eqref{eq: representacion integral de T}, \eqref{eq: hj integra cero contra v} and the smoothness condition on the kernel \eqref{eq:prop del nucleo} on $K$, for $I_3^1$  we get
	\begin{align*}
	I_3^1&\leq\frac{C}{t} \sum_j\int_{\mathbb{R}^n\backslash\Omega^*}|b(x)-b_{Q_j}|T(h_jv)(x)u(x)\,dx\\
	&\leq\frac{C}{t}\sum_j\int_{\mathbb{R}^n\backslash Q_j^*}|b(x)-b_{Q_j}|\left|\int_{Q_j}h_j(y)v(y)(K(x-y)-K(x-x_{Q_j}))\,dy\right|u(x)\,dx\\
	&\leq  \frac{C}{t}\sum_j\int_{Q_j}|h_j(y)|v(y)\int_{\mathbb{R}^n\backslash Q_j^*}|b(x)-b_{Q_j}||K(x-y)-K(x-x_{Q_j})|u_j^*(x)\,dx\,dy\\
&\leq \frac{C}{t}\sum_j\int_{Q_j}|h_j(y)|v(y)\sum_{k=1}^{\infty}{\int_{A_{j,k}}}|b(x)-b_{Q_j}||K(x-y)-K(x-x_{Q_j})|u_j^*(x)\,dx\,dy\\
&= \frac{C}{t}\sum_j\int_{Q_j}|h_j(y)|v(y)\sum_{k=1}^{\infty}{\int_{A_{j,k}}}|b(x)-b_{Q_j}|\frac{|y-x_{Q_j}|}{|x-x_{Q_j}|^{n+1}}u_j^*(x)\,dx\,dy.
	\end{align*}

Observe that there exists a unique $k_0\in\mathbb{N}$ such that $2^{k_0-1}\leq \sqrt{n}<2^{k_0}$. Given $y\in Q_j$ fixed, we get
\begin{align*}
\sum_{k=1}^{\infty}{\int_{A_{j,k}}}|b(x)-b_{Q_j}|\frac{|y-x_{Q_j}|}{|x-x_{Q_j}|^{n+1}}u_j^*(x)\,dx&\leq C\sum_{k=1}^{\infty}\frac{\ell(Q_j)}{2^kr_j}\frac{1}{(2^kr_j)^n}\int_{B\left(x_{Q_j},2^kr_j\right)}|b(x)-b_{Q_j}|u_j^*(x)\,dx\\
&\leq C\sum_{k=1}^{\infty}2^{-k}\frac{1}{|2^{k+k_0+2}Q_j|}\int_{2^{k+k_0+2}Q_j}|b(x)-b_{Q_j}|u_j^*(x)\,dx.
\end{align*}
By applying Lemma~\ref{lema: diferencia de promedios en cubos dilatados acotados por norma BMO}, the generalized Hölder inequality with $\Phi_1$ and $\tilde\Phi_1(\lambda) \approx (e^{\lambda}-1)\mathcal{X}_{(1,\infty)}(\lambda)$, and Lemma~\ref{lema: norma exponencial de la oscilacion media acotada por la norma BMO} we arrive to 
\begin{align*}
\frac{1}{|2^{k+k_0+2}Q_j|}\int_{2^{k+k_0+2}Q_j}|b(x)-b_{Q_j}|u_j^*(x)\,dx&\leq \frac{1}{|2^{k+k_0+2}Q_j|}\int_{2^{k+k_0+2}Q_j}|b(x)-b_{2^{k+k_0+2}Q_j}|u_j^*(x)\,dx\\
&\qquad +C(k+k_0+2)Mu_j^*(y)\\
&\leq C\|b-b_{2^{k+k_0+2}Q_j}\|_{\tilde\Phi_1,2^{k+k_0+2}Q_j}\|u_j^*\|_{\Phi_1,2^{k+k_0+2}Q_j}\\
&\qquad +C(k+k_0+2)Mu_j^*(y)\\
&\leq C(k+k_0+2)M_{\Phi_1^\varepsilon}u_j^*(y)
\end{align*}
This allows us to conclude, by Lemma~\ref{lema: comparacion peso por caracteristica}, that \refstepcounter{BPR2}\label{pag: estimacion de I_3^1}
\begin{align*}
  I_3^1&\leq \frac{C}{t}\sum_j\int_{Q_j}|h_j(y)|M_{\Phi_1^\varepsilon} u_j^*(y) v(y)\,dy\\
  &\leq \frac{C}{t}\sum_j\int_{Q_j}fv\left(\inf_{Q_j}M_{\Phi_1^\varepsilon} u_j^*\right)+C\sum_j\int_{Q_j}f_{Q_j}^vv\left(\inf_{Q_j}M_{\Phi_1^\varepsilon} u_j^*\right),
\end{align*}
and then the desired estimate follows since $M_{\Phi_1^\varepsilon} u_j^*\lesssim M_{\Phi_1^\varepsilon,v^{1-q'}} u$, $v\lesssim M(\Psi(v))$ and $\Phi_1(z)\gtrsim z$.

Finally, we apply Theorem~\ref{teo: mixta tipo F-S para T} in order to estimate $I_3^2$. This yields
\begin{align*}
    I_3^2&=u^*v\left(\left\{x\in \mathbb{R}^n: \left|\sum_j\frac{T((b-b_{Q_j})h_jv)}{v}\right|>\frac{t}{4}\right\}\right)\\
    &=u^*v\left(\left\{x\in \mathbb{R}^n: \left|\frac{T(\sum_j(b-b_{Q_j})h_jv)}{v}\right|>\frac{t}{4}\right\}\right)\\
    &\leq \frac{C}{t}\int_{\mathbb{R}^n}\sum_j|(b-b_{Q_j})h_j|M_{\Phi_1^\varepsilon,v^{1-q'}}u^*M(\Psi(v))\\
    &=\frac{C}{t}\sum_j\int_{Q_j}|(b-b_{Q_j})h_j|M_{\Phi_1^\varepsilon,v^{1-q'}}u_j^*M(\Psi(v)).
\end{align*}

 By applying Lemma~\ref{lema: M(Psi(v)) esencialmente constante sobre cubos} on each $Q_j$, we get
\[I_3^2\lesssim \frac{1}{t}\sum_j \left(\inf_{Q_j}M(\Psi(v))\right)\left(\inf_{Q_j}M_{\varphi_1^\varepsilon,v^{1-q'}}u_j^*\right)\int_{Q_j}|(b-b_{Q_j})h_j|,\]
where we have also used Lemma~\ref{lema: comparacion peso por caracteristica}.

By the generalized Hölder inequality \eqref{eq: desigualdad de Holder generalizada} together with \eqref{eq: equivalencia norma Luxemburgo con infimo} we obtain that \refstepcounter{BPR2}\label{pag: estimacion de I_3^2}
\begin{align*}
    \frac{1}{t}\int_{Q_j}|(b-b_{Q_j})h_j|&\leq \frac{1}{t}\int_{Q_j}|b-b_{Q_j}|f+\frac{1}{t}\int_{Q_j}|b-b_{Q_j}|f_{Q_j}^v\\
    &\leq \frac{1}{t}\int_{Q_j}|b-b_{Q_j}|f+\frac{|Q_j|}{v(Q_j)}\int_{Q_j}\frac{f}{t}v\\
    &\leq C\frac{|Q_j|}{t}\|b-b_{Q_j}\|_{\tilde\Phi_1, Q_j}\|f\|_{\Phi_1,Q_j}+[v]_{\text{RH}_\infty}\int_{Q_j}\frac{f}{t}\\
    &\leq C\frac{|Q_j|}{t}\left(t+\frac{t}{|Q_j|}\int_{Q_j}\Phi_1\left(\frac{f}{t}\right)\right)+[v]_{\text{RH}_\infty}\int_{Q_j}\frac{f}{t}\\
    &\leq C|Q_j|+C\int_{Q_j}\Phi_1\left(\frac{f}{t}\right)\\
    &\leq C\frac{|Q_j|}{v(Q_j)}\int_{Q_j}\frac{f}{t}v+C\int_{Q_j}\Phi_1\left(\frac{f}{t}\right)\\
    &\leq C\int_{Q_j}\Phi_1\left(\frac{f}{t}\right),
\end{align*}
since $v\in \text{RH}_\infty$ and $z\leq \Phi_1(z)$.

By plugging this estimate on the corresponding bound for $I_3^2$ we get 
\[I_3^2\leq C\sum_j \int_{Q_j}\Phi_1\left(\frac{f}{t}\right)M_{\varphi_1^\varepsilon,v^{1-q'}}u\,M(\Psi(v)),\]
which yields the desired estimate. This completes the proof for this case provided $u$ is a bounded function. The estimate for arbitrary $u$ can be performed by using a classic approximation argument. 
\end{proof}

\medskip

\begin{proof}[Proof of Theorem~\ref{teo: mixta tipo F-S para T_b^m}, general case]
We now fix $m$ and assume that the estimate holds for $T_b^k$, for every $1\leq k\leq m-1$. We shall prove that it also holds for $T_b^m$. It will be enough to assume, again, that $u$ is bounded, $f$ is nonnegative and $\|b\|_{\text{BMO}}=1$. We fix $t>0$ and perform the Calderón-Zygmund decomposition of $f$ at level $t$ with respect to $v$, obtaining a collection $\{Q_j\}_{j=1}^\infty$ with the same properties as in the proof above. We also define $\Omega$, $\Omega^*$, $g$ and $h$ as above. Therefore,
\begin{align*}
	uv\left(\left\{x\in \mathbb{R}^n: \left|\frac{T_b^m(fv)}{v}\right|>t\right\}\right)&\leq uv\left(\left\{x\in \mathbb{R}^n\backslash \Omega^*: \left|\frac{T_b^m(gv)}{v}\right|>\frac{t}{2}\right\}\right) + uv(\Omega^*)\\
	&+uv\left(\left\{x\in \mathbb{R}^n\backslash \Omega^*: \left|\frac{T_b^m(hv)}{v}\right|>\frac{t}{2}\right\}\right)\\
	&= I_1^m+I_2^m+I_3^m.
	\end{align*}
	 Fix $\varepsilon>0$, $p>\max\{q,1+(m+1)/\varepsilon\}$ and define $u^*$ and $u_j^*$ as in the proof of the case $m=1$.

In order to deal with $I_1^m$, we combine Tchebycheff inequality with Theorem~\ref{teo: F-S tipo fuerte para conmutador de orden m}, applied with $p'$ and $(m+1)(1-p')+\varepsilon>0$ (from our choice of $p$) and $w=uv^{1-p'}$. We arrive to
\[I_1^m
	\leq\frac{C}{t^{p'}}\int_{\mathbb{R}^n} |T_b^m(gv)|^{p'}u^*v^{1-p'}
	\leq\frac{C}{t^{p'}}\int_{\mathbb{R}^n} (gv)^{p'}M_{\varphi_m^\varepsilon}\left(u^*v^{1-p'}\right).\]
	We can continue by applying \eqref{eq: teo: mixta tipo F-S para T_b^m - eq1} and proceeding as we did in page~\pageref{pag: estimacion de I_1} to get the desired bound for $I_1^m$.
	
	The estimate of $I_2^m$ does not depend on the operator involved so it is the same as the given in page~\ref{pag: estimacion de I_2}. It only remains to estimate $I_3^m$. By applying Lemma~\ref{lema: T_b^m en terminos de conmutadores de orden menor} we can deduce that
	\[T_b^m(hv)(x)=\sum_j(b-b_{Q_j})^mT(h_jv)(x)-\sum_jT((b-b_{Q_j})^mh_jv)(x)-\sum_j\sum_{i=1}^{m-1}C_{m,i}T_b^i((b-b_{Q_j})^{m-i}h_jv)(x)\]
and consequently
\begin{align*}
I_3^m&\leq uv\left(\left\{x\in\mathbb{R}^n\backslash\Omega^*: \left|\frac{\sum_j(b-b_{Q_j})^mT(h_jv)(x)}{v(x)}\right|>\frac{t}{6}\right\}\right)\\
&+uv\left(\left\{x\in\mathbb{R}^n\backslash\Omega^*: \left|\frac{\sum_jT((b-b_{Q_j})^mh_jv)(x)}{v(x)}\right|>\frac{t}{6}\right\}\right)\\
&+uv\left(\left\{x\in\mathbb{R}^n\backslash\Omega^*: \left|\sum_{i=1}^{m-1}\frac{T_b^i(\sum_j(b-b_{Q_j})^{m-i}h_jv)(x)}{v(x)}\right|>\frac{t}{6C}\right\}\right)\\
&=I_3^{m,1}+I_3^{m,2}+I_3^{m,3},
\end{align*}
where $C=\max\{C_{m,i}\}_{i=1}^{m-1}$.

We proceed with the estimate of each term above. Let $A_{j,k}$ and $k_0$ be as in the proof for the case $m=1$. For $I_3^{m,1}$ we apply Tchebycheff inequality in order to get
\begin{align*}
I_3^{m,1}&\leq\frac{6}{t}\int_{\mathbb{R}^n\backslash\Omega^*}\left|\sum_j{(b(x)-b_{Q_j})^mT(h_jv)(x)}\right|u^*(x)\,dx\\
&\leq\frac{C}{t}\sum_j{\int_{\mathbb{R}^n\backslash Q_j^*}|b(x)-b_{Q_j}|^m\left|\int_{Q_j}\hspace*{-0.2cm}\left(K(x-y)-K(x-x_{Q_j})\right)h_j(y)v(y)\,dy\right|u_j^*(x)\,dx}\\
&\leq\frac{C}{t}\sum_j{\int_{Q_j}|h_j(y)|v(y)\int_{\mathbb{R}^n\backslash Q_j^*}|b(x)-b_{Q_j}|^m|K(x-y)-K(x-x_{Q_j})|u_j^*(x)\,dx\,dy}\\
&\leq\frac{C}{t}\sum_j{\int_{Q_j}|h_j(y)|v(y)\sum_{k=1}^{\infty}{\int_{A_{j,k}}|b(x)-b_{Q_j}|^m\frac{|y-x_{Q_j}|}{|x-x_{Q_j}|^{n+1}}u_j^*(x)\,dx\,dy}}\\
&\leq\frac{C}{t}\sum_j\int_{Q_j}|h_j(y)|v(y)\sum_{k=1}^{\infty}\frac{\ell(Q_j)}{2^kr_j}\frac{1}{(2^kr_j)^n}\int_{B\left(x_{Q_j},2^kr_j\right)}|b(x)-b_{Q_j}|^mu_j^*(x)\,dx\\
&\leq \frac{C}{t}\sum_j\int_{Q_j}|h_j(y)|v(y)\sum_{k=1}^{\infty}2^{-k}\frac{1}{|2^{k+k_0+2}Q_j|}\int_{2^{k+k_0+2}Q_j}|b(x)-b_{Q_j}|^mu_j^*(x)\,dx.
\end{align*}

Notice that $\tilde\Phi_m(\lambda)\approx (e^{\lambda^{1/m}}-1)\mathcal{X}_{(1,\infty)}(\lambda)$. By a change of variable we can easily obtain that
\begin{equation}\label{eq: norma exp L y exp L^{1/m}}
\|g_0^m\|_{\tilde\Phi_m, Q}\approx \|g_0\|_{\tilde\Phi_1, Q}^m.
\end{equation}

In order to estimate the inner sum, we apply Lemma~\ref{lema: diferencia de promedios en cubos dilatados acotados por norma BMO}, the generalized Hölder inequality with $\Phi_m$ and $\tilde\Phi_m$ combined with the expression above to conclude
\begin{align*}
\frac{1}{|2^{k+k_0+2}Q_j|}\int_{2^{k+k_0+2}Q_j}|b(x)-b_{Q_j}|^mu_j^*(x)\,dx&\leq \frac{2^m}{|2^{k+k_0+2}Q_j|}\int_{2^{k+k_0+2}Q_j}|b(x)-b_{2^{k+k_0+2}Q_j}|^mu_j^*(x)\,dx\\
&\qquad +C(k+k_0+2)^m Mu_j^*(y)\\
&\leq C\||b-b_{2^{k+k_0+2}Q_j}|^m\|_{\tilde\Phi_{m},2^{k+k_0+2}Q_j}\|u_j^*\|_{\Phi_m,2^{k+k_0+2}Q_j}\\
&\qquad +C(k+k_0+2)^m Mu_j^*(y)\\
&\leq C\|b-b_{2^{k+k_0+2}Q_j}\|_{\tilde\Phi_{1},2^{k+k_0+2}Q_j}^m\|u_j^*\|_{\Phi_m,2^{k+k_0+2}Q_j}\\
&\qquad +C(k+k_0+2)^mMu_j^*(y)\\
&\leq C(k+k_0+2)^mM_{\Phi_m}u_j^*(y).
\end{align*}
This implies that
\[I_3^{m,1}\leq \frac{C}{t}\sum_j\int_{Q_j}|h_j(y)|v(y)M_{\Phi_m^\varepsilon}u_j^*(y)\,dy.\]
From this estimate we can obtain the desired bound for $I_3^{m,1}$ by following the same steps as in page~\pageref{pag: estimacion de I_3^1}.

To deal with $I_3^{m,2}$, we write
\[I_3^{m,2}=u^*v\left(\left\{x\in\mathbb{R}^n: \left|\frac{\sum_jT((b-b_{Q_j})^mh_jv)(x)}{v(x)}\right|>\frac{t}{6}\right\}\right)\]
and apply Theorem~\ref{teo: mixta tipo F-S para T}. We get 
\begin{align*}
I_3^{m,2}&\leq\frac{C}{t}\int_{\mathbb{R}^n}\left|\sum_j{(b(x)-b_{Q_j})^mh_j(x)}\right|M_{\Phi_m^\varepsilon, v^{1-q'}}u^*(x)M(\Psi(v))(x)\,dx\\
&\leq\frac{C}{t}\sum_j{\int_{Q_j}|b(x)-b_{Q_j}|^m|h_j(x)|M_{\Phi_m^\varepsilon, v^{1-q'}}u_j^*(x)M(\Psi(v))(x)\,dx}\\
&\leq\frac{C}{t}\sum_j{\int_{Q_j}|b(x)-b_{Q_j}|^mf(x)M_{\Phi_m^\varepsilon, v^{1-q'}}u_j^*(x)M(\Psi(v))(x)\,dx}\\
&\qquad+\frac{C}{t}\sum_j{\int_{Q_j}|b(x)-b_{Q_j}|^mf_{Q_j}^v M_{\Phi_m^\varepsilon, v^{1-q'}}u_j^*(x)M(\Psi(v))(x)\,dx}\\
&=(\text{A})+(\text{B}).
\end{align*}
By applying generalized Hölder inequality with $\Phi_m$ and $\tilde \Phi_m$ together with \eqref{eq: norma exp L y exp L^{1/m}} and \eqref{eq: equivalencia norma Luxemburgo con infimo} we get \refstepcounter{BPR2}\label{pag: estimacion previa a (A)}
\begin{align*}
\frac{1}{t}\int_{Q_j}|b(x)-b_{Q_j}|^mf(x)\,dx&\lesssim \frac{|Q_j|}{t}\|(b-b_{Q_j})^m\|_{\tilde\Phi_m, Q_j}\|f\|_{\Phi_m, Q_j}\\
&\leq \frac{|Q_j|}{t}\|b-b_{Q_j}\|_{\tilde \Phi_1, Q_j}^m\left(t+\frac{t}{|Q_j|}\int_{Q_j}\Phi_m\left(\frac{f(x)}{t}\right)\,dx\right)\\
&\leq \frac{|Q_j|}{v(Q_j)}\int_{Q_j}\frac{f}{t}v+\int_{Q_j}\Phi_m\left(\frac{f(x)}{t}\right)\\
&\leq C[v]_{\text{RH}_\infty}\int_{Q_j}\Phi_m\left(\frac{f(x)}{t}\right).
\end{align*}
Therefore, by combining Lemma~\ref{lema: comparacion peso por caracteristica} with Lemma~\ref{lema: M(Psi(v)) esencialmente constante sobre cubos}, we arrive to
\begin{align*}
(\text{A})&\leq \frac{C[v]_{\text{RH}_\infty}}{t}\sum_j \left(\inf_{Q_j} M_{\Phi_m^\varepsilon, v^{1-q'}}u_j^*\right)\left(\inf_{Q_j} M(\Psi(v))\right)\int_{Q_j}\Phi_m\left(\frac{f(x)}{t}\right)\,dx\\
&\leq C\sum_j\int_{Q_j}\Phi_m\left(\frac{f(x)}{t}\right)M_{\Phi_m^\varepsilon, v^{1-q'}}u_j^*(x)M(\Psi(v))(x)\,dx\\
&\leq C\int_{\mathbb{R}^n}\Phi_m\left(\frac{f(x)}{t}\right)M_{\Phi_m^\varepsilon, v^{1-q'}}u(x)M(\Psi(v))(x)\,dx.
\end{align*}
For $(\text{B})$, we proceed similarly as for $(\text{A})$, after noticing that
\[\frac{f_{Q_j}^v}{t}\int_{Q_j}|b(x)-b_{Q_j}|^m\lesssim \frac{|Q_j|}{v(Q_j)}\int_{Q_j}\frac{f}{t}v\leq [v]_{\text{RH}_\infty}\int_{Q_j}\frac{f}{t}\leq [v]_{\text{RH}_\infty}\int_{Q_j}\Phi_m\left(\frac{f}{t}\right).\]
It only remains to estimate $I_3^{m,3}$. By applying the inductive hypothesis for the lower order commutators we get
\begin{align*}
I_3^{m,3}&\leq \sum_{i=1}^{m-1}u^*v\left(\left\{x\in\mathbb{R}^n: \left|\frac{T_b^i(\sum_j(b-b_{Q_j})^{m-i}h_jv)(x)}{v(x)}\right|>\frac{t}{6C}\right\}\right)\\
&\lesssim \sum_{i=1}^{m-1} \int_{\mathbb{R}^n}\Phi_i\left(\frac{\sum_j |(b(x)-b_{Q_j})^{m-i}h_j(x)|}{t}\right)M_{\Phi_i^\varepsilon, v^{1-q'}}u^*(x)M(\Psi(v))(x)\,dx\\
&\leq  \sum_{i=1}^m\sum_j \int_{Q_j}\Phi_i\left(\frac{|(b(x)-b_{Q_j})^{m-i}h_j(x)|}{t}\right)M_{\Phi_i^\varepsilon, v^{1-q'}}u_j^*(x)M(\Psi(v))(x)\,dx\\
&\lesssim \sum_{i=1}^m\sum_j \left(\inf_{Q_j}M_{\Phi_i^\varepsilon, v^{1-q'}}u_j^*\right)\left(\inf_{Q_j}M(\Psi(v))\right)\int_{Q_j}\Phi_i\left(\frac{|(b(x)-b_{Q_j})^{m-i}h_j(x)|}{t}\right)\,dx,
\end{align*}
by virtue of Lemma~\ref{lema: comparacion peso por caracteristica} and Lemma~\ref{lema: M(Psi(v)) esencialmente constante sobre cubos}. Since\refstepcounter{BPR2}\label{pag: estimacion de I_3^{m,3}}
\[\Phi^{-1}_i(\lambda)\approx \frac{\lambda}{(1+\log^+\lambda)^i}\]
for $1\leq i\leq m$, if we denote $\psi_i(\lambda)=e^{\lambda^{1/i}}-1$, we get
\[\Phi_m^{-1}(\lambda)\psi_{m-i}^{-1}(\lambda)\approx \frac{\lambda}{(1+\log^+\lambda)^m}\log(1+\lambda)^{m-i}\approx \frac{\lambda}{(1+\log^+\lambda)^i}\approx \Phi_i^{-1}(\lambda),\]
and 
consequently,
\begin{align*}
\int_{Q_j}\Phi_i\left(\frac{|(b(x)-b_{Q_j})^{m-i}h_j(x)|}{t}\right)\,dx&\leq \int_{Q_j}\Phi_m\left(\frac{|h_j(x)|}{t}\right)\,dx+\int_{Q_j}\psi_{m-i}\left(|b(x)-b_{Q_j}|^{m-i}\right)\,dx.
\end{align*}
Observe that, since $\Phi_m$ is a convex function
\begin{align*}
\int_{Q_j}\Phi_m\left(\frac{|h_j(x)|}{t}\right)\,dx&\lesssim \int_{Q_j}\Phi_m\left(\frac{f(x)}{t}\right)\,dx+\int_{Q_j}\Phi_m\left(\frac{f_{Q_j}^v}{t}\right)\,dx\\
&\lesssim \int_{Q_j}\Phi_m\left(\frac{f(x)}{t}\right)\,dx+|Q_j|\\
&\lesssim [v]_{\mathrm{RH}_\infty}\int_{Q_j}\Phi_m\left(\frac{f(x)}{t}\right)\,dx.
\end{align*}
On the other hand, if $\psi(\lambda)=e^{\lambda}-1$, by virtue of Lemma~\ref{lema: norma exponencial de la oscilacion media acotada por la norma BMO} we have that
\begin{align*}
\int_{Q_j}\psi_{m-i}\left(|b(x)-b_{Q_j}|^{m-i}\right)\,dx&= \int_{Q_j}\psi\left(|b(x)-b_{Q_j}|\right)\,dx\\
&\leq \int_{Q_j}\psi\left(\frac{|b(x)-b_{Q_j}|}{\|b-b_{Q_j}\|_{\psi, Q_j}}\right)\,dx\\
&\leq |Q_j|\\
&\lesssim [v]_{\mathrm{RH}_\infty}\int_{Q_j}\frac{f(x)}{t}\,dx.
\end{align*}
These estimates allow us to get the desired bound for $I_3^{m,3}$. The proof is complete.
\end{proof}

\section{Proof of Theorem~\ref{teo: mixta tipo F-S para T_b^m Hormander}}\label{seccion: F-S para Hormander}
We conclude the article with the proof of the mixed inequality of Fefferman-Stein type for operators associated to kernels with less regularity. The following theorem, proved in \cite{LRdlT}, establishes a Coifman type estimate for commutators of these classes of operators. We shall require this estimate for the main proof in this section.

\begin{teo}\label{teo: tipo Coifman para Tb^m hormander}
Let $m\in\mathbb{N}$. Let $\xi$ and $\zeta$ two Young functions such that $\tilde\xi^{-1}(t)\zeta^{-1}(t)(\log t)^m \lesssim t$, for every $t\geq t_0\geq e$. If $T$ is an operator as in \eqref{eq: representacion integral de T} with an associated kernel $K\in H_{\zeta}\cap H_{\xi,m}$ and $b\in \text{BMO}$, then for every $0<p<\infty$ and $w\in A_\infty$ the inequality
\[\int_{\mathbb{R}^n} |T_b^mf(x)|^pw(x)\,dx\leq C\|b\|_{\text{BMO}}^{mp}\int_{\mathbb{R}^n}M_{\tilde \xi}f(x)^pw(x)\,dx\]
holds for every bounded function $f$ with compact support, provided the left-hand side is finite.
\end{teo}

As we did in the previous section, we shall proceed by induction on $m$, so we separate the cases corresponding to $m=1$ and $m>1$.

\begin{proof}[Proof of Theorem~\ref{teo: mixta tipo F-S para T_b^m Hormander}, case $m=1$]
As in the proof of Theorem~\ref{teo: mixta tipo F-S para T_b^m}, we assume that $u$ is bounded, $f$ is positive and $\|b\|_{\mathrm{BMO}}=1$.
We fix $t>0$ and perform the Calder\'on-Zygmund decomposition of $f$ at level $t$ with respect to $v$, obtaining the collection $\{Q_j\}_{j=1}^\infty$ satisfying $t<f_{Q_j}^v\leq Ct$. We also write $f=g+h$ as before, and define $Q_j^*=4c\sqrt{n}Q_j$, where $c$ is the constant appearing in \eqref{eq: condicion Hormander - m}. We split as before
	\begin{align*}
	uv\left(\left\{x\in \mathbb{R}^n: \left|\frac{T_b(fv)}{v}\right|>t\right\}\right)&\leq uv\left(\left\{x\in \mathbb{R}^n\backslash \Omega^*: \left|\frac{T_b(gv)}{v}\right|>\frac{t}{2}\right\}\right) + uv(\Omega^*)\\
	&+uv\left(\left\{x\in \mathbb{R}^n\backslash \Omega^*: \left|\frac{T_b(hv)}{v}\right|>\frac{t}{2}\right\}\right)\\
	&= I_1+I_2+I_3.
	\end{align*}
 For $I_1$, we apply Tchebycheff inequality with $p'>1$ in order to get
\[I_1\leq\frac{C}{t^{p'}}\int_{\mathbb{R}^n} |T_b(gv)|^{p'}uv^{1-p'}\mathcal{X}_{\mathbb{R}^n\backslash \Omega^*}
	\leq\frac{C}{t^{p'}}\int_{\mathbb{R}^n} |T_b(gv)|^{p'}M_s\left(u^*v^{1-p'}\right).\]
In order to apply Theorem~\ref{teo: tipo Coifman para Tb^m hormander}, we must show that $\|T_b(gv)\|_{L^{p'}(w)}<\infty$, where $w=M_s(u^*v^{1-p'})$ is an $A_1$ weight. Since the hypotheses imply that $K\in H_1$, we have that $T$ is bounded on $L^{p'}(\mathbb{R}^n)$ (see \cite{javi}). Therefore a classical argument allows to show that this quantity is finite (see pages 1415--1416 in \cite{LRdlT}). 
Consequently, we combine  Theorem~\ref{teo: tipo Coifman para Tb^m hormander} with the lower type $s$ and the upper type $r$ of $\tilde\xi$ to obtain
\begin{align*}
I_1&\lesssim\frac{1}{t^{p'}}\int_{\mathbb{R}^n} \left(M_{\tilde\xi}(gv)\right)^{p'}M_s\left(u^*v^{1-p'}\right)\\
&\lesssim\frac{1}{t^{p'}}\int_{\mathbb{R}^n} M_r(gv)^{p'}M_{s}\left(u^*v^{1-p'}\right)\\
 &\lesssim\frac{1}{t^{p'}}\int_{\mathbb{R}^n} (gv)^{p'}M_{\tilde\xi}\left(u^*v^{1-p'}\right).
\end{align*}
By applying Lemma~\ref{lema: estimacion de maximal del producto de pesos - 2}, we can conclude that
	\begin{equation}\label{eq: estimacion de maximal del producto de pesos - 2}
	M_{\tilde\xi}\left(u^*v^{1-p'}\right)(x)\lesssim M_{\varphi_p,v^{1-q'}}(u^*)(x)v^{-p'}(x)\Psi(v(x)),
	\end{equation}
	 for almost every $x$. This yields
 \[I_1\lesssim \frac{1}{t^{p'}}\int_{\mathbb{R}^n}(gv)^{p'}\left(M_{\varphi_p, v^{1-q'}}u^*\right)v^{-p'}\Psi(v)\]
 and we can deduce the desired estimate for $I_1$ from the inequality above by proceeding as in page~\pageref{pag: estimacion de I_1}.

 For $I_2$, we can proceed as in page~\pageref{pag: estimacion de I_2} since we have $\lambda\lesssim \varphi_p(\lambda)$.
	 
In order to deal with $I_3$, we write 
	\begin{align*}
	I_3&\leq uv\left(\left\{x\in \mathbb{R}^n\backslash \Omega^*: \left|\sum_j\frac{(b-b_{Q_j})T(h_jv)}{v}\right|>\frac{t}{4}\right\}\right)\\
	&\qquad +uv\left(\left\{x\in \mathbb{R}^n\backslash \Omega^*: \left|\sum_j\frac{T((b-b_{Q_j})h_jv)}{v}\right|>\frac{t}{4}\right\}\right)\\
	&=I_3^1+I_3^2.
	\end{align*}
	
	In this case we set $A_{j,k}=\{x: 2^{k-1}r_j<|x-x_{Q_j}|\leq 2^{k}r_j\}$, where $r_j=2c\sqrt{n}\ell(Q_j)$, being $c$ the constant appearing in \eqref{eq: condicion Hormander - m}. By using the integral representation of $T$ in \eqref{eq: representacion integral de T} and \eqref{eq: hj integra cero contra v} we obtain
 \begin{align*}
 I_3^1&\lesssim \frac{1}{t}\sum_j\int_{Q_j}|h_j(y)|v(y)\sum_{k=1}^{\infty}{\int_{A_{j,k}}}|b(x)-b_{Q_j}||K(x-y)-K(x-x_{Q_j})|u_j^*(x)\,dx\,dy\\
 &=\frac{1}{t}\sum_j\int_{Q_j}|h_j(y)|v(y)F_j(y)\,dy,
 \end{align*}
 where $u_j^*=u\mathcal{X}_{\mathbb{R}^n\backslash Q_j^*}$.
 
 Notice that there exists a unique $k_0\in\mathbb{N}$ such that $2^{k_0-1}\leq c\sqrt{n}<2^{k_0}$. We have that
 \begin{align*}
 F_j(y)&\leq \sum_{k=1}^\infty \int_{A_{j,k}}|b(x)-b_{2^{k+k_0+2}Q_j}||K(x-y)-K(x-x_{Q_j})|u_j^*(x)\,dx\\
 &\qquad + \sum_{k=1}^\infty \int_{A_{j,k}}|b_{Q_j}-b_{2^{k+k_0+2}Q_j}||K(x-y)-K(x-x_{Q_j})|u_j^*(x)\,dx\\
 &=F_j^1(y)+F_j^2(y).
 \end{align*}
 For $F_j^1$ we apply the generalized Hölder inequality \eqref{eq: Holder generalizada con promedios Luxemburgo} with the functions $\tilde\xi, \zeta$ and $\psi(\lambda)=e^{\lambda}-1$. Since $K\in H_\zeta$ and $B(x_{Q_j},2^kr_j)\subset 2^{k+k_0+2}Q_j$, by \eqref{eq: condicion Hormander} we get
\begin{align*}
F_j^1(y)&\leq \sum_{k=1}^\infty (2^kr_j)^n\|b-b_{2^{k+k_0+2}Q_j}\|_{\psi, 2^{k+k_0+2}Q_j}\|K(\cdot-(y-x_{Q_j}))-K(\cdot)\|_{\zeta, |x-x_{Q_j}|\sim 2^kr_j}\|u_j^*\|_{\tilde\xi, 2^{k+k_0+2}Q_j}\\
&\leq C_{\zeta} M_{\tilde\xi}\, u_j^*(y)\\
&\lesssim M_{\varphi_p, v^{1-q'}} u_j^*(y),
\end{align*}
where we have used again the fact that $\tilde \xi(\lambda)\lesssim \varphi_p(\lambda)$ for large $\lambda$.

For $F_j^2$, we use the generalized Hölder inequality \eqref{eq: desigualdad de Holder generalizada}, Lemma~\ref{lema: diferencia de promedios en cubos dilatados acotados por norma BMO} and the fact that $K\in H_{\xi,1}\cap H_{\xi}$ in order to get
\begin{align*}
F_j^2&\leq \sum_{k=1}^\infty (2^kr_j)^n(k+k_0+2)\|K(\cdot-(y-x_{Q_j}))-K(\cdot)\|_{\xi, |x-x_{Q_j}|\sim 2^kr_j}\|u_j^*\|_{\tilde\xi, 2^{k+k_0+2}Q_j}\\
&\lesssim C_{\xi} M_{\varphi_p,v^{1-q'}}u_j^*(y).
\end{align*}

These estimates allow us to conclude that
\[I_3^1\lesssim \frac{1}{t}\sum_j \int_{Q_j}|h_j(y)|v(y)M_{\varphi_p, v^{1-q'}}u_j^*(y)\,dy.\]
From this inequality we can continue by proceeding as in page \pageref{pag: estimacion de I_3^1}.
 
Finally, since the hypotheses imply those in Theorem~\ref{teo: F-S para T Hormander}, we can  estimate $I_3^2$ as follows
\begin{align*}
    I_3^2&=u^*v\left(\left\{x\in \mathbb{R}^n: \left|\sum_j\frac{T((b-b_{Q_j})h_jv)}{v}\right|>\frac{t}{4}\right\}\right)\\
    &=u^*v\left(\left\{x\in \mathbb{R}^n: \left|\frac{T(\sum_j(b-b_{Q_j})h_jv)}{v}\right|>\frac{t}{4}\right\}\right)\\
    &\leq \frac{C}{t}\int_{\mathbb{R}^n}\sum_j|(b-b_{Q_j})h_j|M_{\varphi_p,v^{1-q'}}u^*M(\Psi(v))\\
    &\leq\frac{C}{t}\sum_j\int_{Q_j}|(b-b_{Q_j})h_j|M_{\varphi_p,v^{1-q'}}u_j^*M(\Psi(v)).
\end{align*}
From the estimate above, we can repeat the steps in page~\pageref{pag: estimacion de I_3^2} in order to obtain the desired bound. This completes the proof for the case $m=1$.
\end{proof}

\medskip

\begin{proof}[Proof of Theorem~\ref{teo: mixta tipo F-S para T_b^m Hormander}, general case]
We assume that the estimate holds for $T_b^k$, for every $1\leq k\leq m-1$ and prove that it also holds for $T_b^m$.  We  perform the Calderón-Zygmund decomposition of $f$ at level $t>0$ fixed with respect to $v$, obtaining a collection $\{Q_j\}_{j=1}^\infty$ with the same properties as in the proof of the previous case. We also define $\Omega$, $\Omega^*$, $g$ and $h$ as before. Then
\begin{align*}
	uv\left(\left\{x\in \mathbb{R}^n: \left|\frac{T_b^m(fv)}{v}\right|>t\right\}\right)&\leq uv\left(\left\{x\in \mathbb{R}^n\backslash \Omega^*: \left|\frac{T_b^m(gv)}{v}\right|>\frac{t}{2}\right\}\right) + uv(\Omega^*)\\
	&+uv\left(\left\{x\in \mathbb{R}^n\backslash \Omega^*: \left|\frac{T_b^m(hv)}{v}\right|>\frac{t}{2}\right\}\right)\\
	&= I_1^m+I_2^m+I_3^m.
	\end{align*}
For $I_1^m$, we apply Tchebycheff inequality with $p'$ in order to get
\[I_1^m
	\leq\frac{C}{t^{p'}}\int_{\mathbb{R}^n} |T_b^m(gv)|^{p'}u^*v^{1-p'}\leq \frac{C}{t^{p'}}\int_{\mathbb{R}^n} |T_b^m(gv)|^{p'}M_s\left(u^*v^{1-p'}\right).\]
 Recall that $T$ is bounded on $L^2(\mathbb{R}^n)$, so we can repeat the corresponding argument given in the case $m=1$ to show that the right-hand side above is finite. Consequently, by applying Theorem~\ref{teo: tipo Coifman para Tb^m hormander} we obtain
\begin{align*}
I_1^m&\lesssim \frac{1}{t^{p'}}\int_{\mathbb{R}^n} \left(M_{\tilde\xi}(gv)\right)^{p'}M_s\left(u^*v^{1-p'}\right)\\
&\lesssim\frac{1}{t^{p'}}\int_{\mathbb{R}^n} M_r(gv)^{p'}M_{s}\left(u^*v^{1-p'}\right)\\
 &\lesssim\frac{1}{t^{p'}}\int_{\mathbb{R}^n} (gv)^{p'}M_{\tilde\xi}\left(u^*v^{1-p'}\right),
\end{align*}
where we have used both the lower and the upper type of $\tilde\xi$ and the fact that $M_r$ is bounded on $L^{p'/r}(w)$ with $w\in A_1$, since $p'>r$.
 
From this estimate, we can apply inequality \eqref{eq: estimacion de maximal del producto de pesos - 2} given by Lemma~\ref{lema: estimacion de maximal del producto de pesos - 2} and proceed in the same manner as we did before in page~\pageref{pag: estimacion de I_1} to get the desired bound for $I_1^m$.

The estimate of $I_2^m$ does not involve the commutator operator. Recall that our hypotheses imply that $v^{q'}$ belongs to $\mathrm{RH}_\infty$ and $v^{1-q'}\in A_1$. Therefore it can be achieved following the same steps as in page~\pageref{pag: estimacion de I_2}.

 It only remains to estimate $I_3^m$. By Lemma~\ref{lema: T_b^m en terminos de conmutadores de orden menor} we write
	\[T_b^m(hv)(x)=\sum_j(b-b_{Q_j})^mT(h_jv)(x)-\sum_jT((b-b_{Q_j})^mh_jv)(x)-\sum_j\sum_{i=1}^{m-1}C_{m,i}T_b^i((b-b_{Q_j})^{m-i}h_jv)(x),\]
 so we decompose
\begin{align*}
I_3^m&\leq uv\left(\left\{x\in\mathbb{R}^n\backslash\Omega^*: \left|\frac{\sum_j(b-b_{Q_j})^mT(h_jv)(x)}{v(x)}\right|>\frac{t}{6}\right\}\right)\\
&+uv\left(\left\{x\in\mathbb{R}^n\backslash\Omega^*: \left|\frac{\sum_jT((b-b_{Q_j})^mh_jv)(x)}{v(x)}\right|>\frac{t}{6}\right\}\right)\\
&+uv\left(\left\{x\in\mathbb{R}^n\backslash\Omega^*: \left|\sum_{i=1}^{m-1}\frac{T_b^i(\sum_j(b-b_{Q_j})^{m-i}h_jv)(x)}{v(x)}\right|>\frac{t}{6C}\right\}\right)\\
&=I_3^{m,1}+I_3^{m,2}+I_3^{m,3},
\end{align*}
where $C=\max\{C_{m,i}\}_{i=1}^{m-1}$.

We need to estimate every term above. By using the integral representation \eqref{eq: representacion integral de T} together with \eqref{eq: hj integra cero contra v} we get
\[I_3^{m,1}\lesssim\frac{1}{t}\sum_j{\int_{Q_j}|h_j(y)|v(y)\int_{\mathbb{R}^n\backslash Q_j^*}|b(x)-b_{Q_j}|^m|K(x-y)-K(x-x_{Q_j})|u_j^*(x)\,dx\,dy}.\]

Let $r_j$, $A_{j,k}$ and $k_0$ be as in the proof for the case $m=1$. We have that 
\begin{align*}
I_3^{m,1}
&\lesssim\frac{1}{t}\sum_j{\int_{Q_j}|h_j(y)|v(y)\sum_{k=1}^\infty\int_{A_{j,k}}|b(x)-b_{Q_j}|^m|K(x-y)-K(x-x_{Q_j})|u_j^*(x)\,dx\,dy}\\
&\lesssim\frac{1}{t}\sum_j{\int_{Q_j}|h_j(y)|v(y)\sum_{k=1}^\infty\int_{A_{j,k}}|b(x)-b_{2^{k+k_0+2}Q_j}|^m|K(x-y)-K(x-x_{Q_j})|u_j^*(x)\,dx\,dy}\\
&\qquad + \frac{1}{t}\sum_j{\int_{Q_j}|h_j(y)|v(y)\sum_{k=1}^\infty\int_{A_{j,k}}|b_{2^{k+k_0+2}Q_j}-b_{Q_j}|^m|K(x-y)-K(x-x_{Q_j})|u_j^*(x)\,dx\,dy}\\
&=\frac{1}{t}\sum_j \int_{Q_j}|h_j(y)|v(y)\left(F_{j,m}^1(y)+F_{j,m}^2(y)\right)\,dy.
\end{align*}
We estimate the terms $F_{j,m}^1$ and $F_{j,m}^2$ separately. Let $\psi_m(\lambda)=e^{\lambda^{1/m}}-1$. From the hypothesis $\tilde\xi^{-1}(\lambda)\zeta^{-1}(\lambda)(\log(\lambda))^m\lesssim \lambda$, we apply the generalized Hölder inequality with $\psi_m$, $\zeta$ and $\tilde\xi$ in order to obtain 
\begin{align*}
F_{j,m}^1(y)&\lesssim \sum_{k=1}^\infty (2^kr_j)^n\||b-b_{2^{k+k_0+2}Q_j}|^m\|_{\psi_m, 2^{k+k_0+2}Q_j}\|K(\cdot-(y-x_{Q_j}))-K(\cdot)\|_{\zeta, |x-x_{Q_j}|\sim 2^kr_j}\|u_j^*\|_{\tilde\xi, 2^{k+k_0+2}Q_j}\\
&\lesssim  M_{\tilde\xi}\,u_j^*(y)\sum_{k=1}^\infty (2^kr_j)^n\|b-b_{2^{k+k_0+2}Q_j}\|_{\psi_1, 2^{k+k_0+2}Q_j}\|K(\cdot-(y-x_{Q_j}))-K(\cdot)\|_{\zeta, |x-x_{Q_j}|\sim 2^kr_j}\\
&\leq C_{\zeta} M_{\tilde\xi}\,u_j^*(y),
\end{align*}
where we have used \eqref{eq: norma exp L y exp L^{1/m}}, Lemma~\ref{lema: norma exponencial de la oscilacion media acotada por la norma BMO} and condition \eqref{eq: condicion Hormander}, since $K\in H_{\zeta}$. 

For $F_{j,m}^2$ we apply Lemma~\ref{lema: diferencia de promedios en cubos dilatados acotados por norma BMO}, the generalized Hölder inequality with $\xi$ and $\tilde \xi$ and the fact that $K\in H_{\xi,m}$ to get
\begin{align*}
F_{j,m}^2(y)&\lesssim \sum_{k=1}^\infty (2^kr_j)^n(k+k_0+2)^m\|K(\cdot-(y-x_{Q_j}))-K(\cdot)\|_{\xi, |x-x_{Q_j}|\sim 2^kr_j}\|u_j^*\|_{\tilde\xi, 2^{k+k_0+2}Q_j}\\
&\leq C_{m,\xi} M_{\tilde\xi}\,u_j^*(y).
\end{align*}

These two estimates imply that
\[I_3^{m,1}\lesssim \frac{1}{t}\sum_j\int_{Q_j} |h_j(y)|v(y)M_{\tilde\xi} u_j^*(y)\,dy.\]
The desired bound for $I_3^{m,1}$ can now be achieved by following the same steps as in page~\pageref{pag: estimacion de I_3^1}. 

In order to estimate $I_3^{m,2}$, we observe that
\[I_3^{m,2}=u^*v\left(\left\{x\in\mathbb{R}^n: \left|\frac{\sum_jT((b-b_{Q_j})^mh_jv)(x)}{v(x)}\right|>\frac{t}{6}\right\}\right)\]
and apply Theorem~\ref{teo: F-S para T Hormander} to obtain that 
\begin{align*}
I_3^{m,2}&\leq\frac{C}{t}\int_{\mathbb{R}^n}\left|\sum_j{(b(x)-b_{Q_j})^mh_j(x)}\right|M_{\varphi_p, v^{1-q'}}u^*(x)M(\Psi(v))(x)\,dx\\
&\lesssim\frac{1}{t}\sum_j{\int_{Q_j}|b(x)-b_{Q_j}|^m|h_j(x)|M_{\varphi_p, v^{1-q'}}u_j^*(x)M(\Psi(v))(x)\,dx}\\
&\leq\frac{1}{t}\sum_j{\int_{Q_j}|b(x)-b_{Q_j}|^mf(x)M_{\varphi_p, v^{1-q'}}u_j^*(x)M(\Psi(v))(x)\,dx}\\
&\qquad+\frac{1}{t}\sum_j{\int_{Q_j}|b(x)-b_{Q_j}|^mf_{Q_j}^v M_{\varphi_p, v^{1-q'}}u_j^*(x)M(\Psi(v))(x)\,dx}.
\end{align*}
By virtue of Lemma~\ref{lema: comparacion peso por caracteristica} and Lemma~\ref{lema: M(Psi(v)) esencialmente constante sobre cubos}, we can continue as we have previously done in page~\pageref{pag: estimacion previa a (A)}, so the desired estimate follows.

We now proceed with $I_3^{m,3}$ in order to conclude. By the inductive hypothesis applied to the lower order commutators we have 
\begin{align*}
I_3^{m,3}&\leq \sum_{i=1}^{m-1}u^*v\left(\left\{x\in\mathbb{R}^n: \left|\frac{T_b^i(\sum_j(b-b_{Q_j})^{m-i}h_jv)(x)}{v(x)}\right|>\frac{t}{6C}\right\}\right)\\
&\lesssim \sum_{i=1}^{m-1} \int_{\mathbb{R}^n}\Phi_i\left(\frac{\sum_j |(b(x)-b_{Q_j})^{m-i}h_j(x)|}{t}\right)M_{\varphi_p, v^{1-q'}}u^*(x)M(\Psi(v))(x)\,dx\\
&\leq  \sum_{i=1}^m\sum_j \int_{Q_j}\Phi_i\left(\frac{|(b(x)-b_{Q_j})^{m-i}h_j(x)|}{t}\right)M_{\varphi_p, v^{1-q'}}u_j^*(x)M(\Psi(v))(x)\,dx\\
&\lesssim \sum_{i=1}^m\sum_j \left(\inf_{Q_j}M_{\varphi_p, v^{1-q'}}u_j^*\right)\left(\inf_{Q_j}M(\Psi(v))\right)\int_{Q_j}\Phi_i\left(\frac{|(b(x)-b_{Q_j})^{m-i}h_j(x)|}{t}\right)\,dx,
\end{align*}
according to Lemma~\ref{lema: comparacion peso por caracteristica} and Lemma~\ref{lema: M(Psi(v)) esencialmente constante sobre cubos}. Now we continue as we did in page~\pageref{pag: estimacion de I_3^{m,3}} in order to get the desired bound. This completes the proof.
\end{proof}

\section*{Declarations}

\subsection*{Ethical approval}
Not applicable.
\subsection*{Competing interests}
Not applicable.
\subsection*{Author's contributions}
Not applicable.
\subsection*{Funding}
The authors were supported by PICT 2019 Nº 389 (ANPCyT),  CAI+D 2020 50320220100210 (UNL) and PICT 2018 Nº 02501.
\subsection*{Availability of data and materials}
Not applicable.

\def\cprime{$'$}
\providecommand{\bysame}{\leavevmode\hbox to3em{\hrulefill}\thinspace}
\providecommand{\MR}{\relax\ifhmode\unskip\space\fi MR }
\providecommand{\MRhref}[2]{%
  \href{http://www.ams.org/mathscinet-getitem?mr=#1}{#2}
}
\providecommand{\href}[2]{#2}

\end{document}